\documentclass{amsart}
\usepackage{amsmath}
\usepackage{amssymb}
\usepackage{graphicx}
\usepackage{color}
\usepackage{float}
\usepackage[all]{xy}
\usepackage{amsthm}
\usepackage{bbm}
\usepackage[shortlabels]{enumitem}
\usepackage[colorlinks]{hyperref}
\usepackage{tikz}
\usetikzlibrary{positioning}
\usetikzlibrary{decorations,arrows,shapes}
\usepackage{tikz-cd}

\theoremstyle{plain}
\newtheorem{thm}{Theorem}[section]
\newtheorem{prop}[thm]{Proposition}
\newtheorem{lem}[thm]{Lemma}
\newtheorem{cor}[thm]{Corollary}

\newtheorem{example}[thm]{Example}

\theoremstyle{remark}

\newtheorem{remark}[thm]{Remark}

\theoremstyle{definition}
\newtheorem{dfn}[thm]{Definition}

\newcommand{\N}{\mathbb{N}}
\newcommand{\Z}{\mathbb{Z}}

\newcommand{\T}{\mathbb{T}}
\newcommand{\Aut}{\operatorname{Aut}}

\newcommand{\calC}{\mathcal{C}}
\newcommand{\calD}{\mathcal{D}}

\newcommand{\lact}{\vartriangleright}
\newcommand{\ract}{\vartriangleleft}

\newcommand{\into}{\hookrightarrow}

\newcommand{\id}{\operatorname{id}}

\definecolor{deepblue}{RGB}{0,19,222}

\title{Products, crossed products, and Zappa--Sz\'ep products for $k$-graphs}

\author[A. Abell-Ball]{Adlin Abell-Ball}
\address{Adlin Abell-Ball,  Department of Mathematical Sciences, University of Montana, Missoula, MT 59812-0864, USA} 
\email{adlinfaball@gmail.com}

\author[E. Gillaspy]{Elizabeth Gillaspy}
\address{Elizabeth Gillaspy,  Department of Mathematical Sciences, University of Montana, Missoula, MT 59812-0864, USA} \email{elizabeth.gillaspy@mso.umt.edu}

\author[G. Glidden-Handgis]{George Glidden-Handgis}
\address{George Glidden-Handgis, University of Arizona College of Pharmacy, 
Tucson, AZ 85721, USA}
\email{georgeglidden@arizona.edu}

\author[S.J. Lippert]{S. Joseph Lippert}
\address{S. Joseph Lippert, Department of Mathematics and Statistics, Sam Houston State University, Huntsville, Texas 77341, USA}
\email{sjl054@shsu.edu}

\date{\today}

\thanks{This research was supported by the US National Science Foundation (grant DMS-1800749 to E.G.).}



\begin{document}

\begin{abstract}
    We use the lens of Zappa--Sz\'ep decomposition to examine the relationship between directed graph products and $k$-graph products.  
    There are many examples of higher-rank graphs, or $k$-graphs, whose underlying directed graph may be factored as a product, but the $k$-graph itself is not a product.
    In such examples, we establish that 
    the Zappa--Sz\'ep structure of the $k$-graph  
    gives rise to ``actions'' of the underlying directed factors on each other.  
    Although these ``actions'' are in general poorly behaved, 
    if one of  them 
    is trivial (or trivial up to isomorphism), we obtain  
    a crossed-product-like structure on the $k$-graph. We 
    provide examples where this crossed-product structure is 
    visible in the associated $C^*$-algebra, and we characterize those $k$-graphs whose Zappa--Sz\'ep induced actions are trivial up to isomorphism.   
\end{abstract}
\maketitle 
\section{Introduction}

Graph $C^*$-algebras have provided key examples and inspiration for $C^*$-algebra theory, since the underlying directed graph determines much of the structure of the associated $C^*$-algebra.  For example, the $K$-theory \cite{raeburn-szyman} and the lattice of ideals \cite{bhrs} of a graph $C^*$-algebra are completely determined by the underlying graph; and for finite graphs $E, F$, there is a purely graph-theoretic description \cite{compclass} of when $C^*(E) \cong C^*(F)$. 
This close link between the graphical  and the $C^*$-algebraic structure also implies that many important examples of $C^*$-algebras (such as the irrational rotation algebras \cite{evans-sims}) 
do not arise as graph $C^*$-algebras.

Higher-rank graphs, or $k$-graphs, were introduced by Kumjian and Pask in \cite{kp} 
to provide
computable, combinatorial models of a broader class of $C^*$-algebras.  
 While the class of $k$-graph $C^*$-algebras includes many  examples which are not graph $C^*$-algebras (cf.~\cite{prrs, pss-rr0, stable-exotic}), constructing concrete examples of $k$-graphs, and computing the associated $C^*$-algebras, is often surprisingly subtle.

One straightforward way to create examples of $k$-graphs is the product $k$-graph construction.  
Higher-rank graphs form a monoidal category; in particular, the (Cartesian) product of two higher-rank graphs is again a higher-rank graph.  Moreover \cite[Corollary 3.5(iv)]{kp}, the $C^*$-algebra of a product $k$-graph $\Omega = \Lambda \times \Gamma$ is the tensor product of the factor $C^*$-algebras, $C^*(\Omega) \cong C^*(\Lambda) \otimes C^*(\Gamma)$.

Our goal in this paper is to  identify generalizations of the product $k$-graph construction in which, as in the product graph case, the $k$-graph  structure encodes $C^*$-algebraic structure.  To do this we use the framework of Zappa--Sz\'ep products.
 
A Zappa--Sz\'ep product $G \bowtie H$ is a generalization of a semidirect product $G \rtimes H$ which encodes compatible actions of $G$ on $H$ and of $H$ on $G$.  Originally studied for groups \cite{zappa, szep}, Zappa--Sz\'ep products of semigroups \cite{kunze, brin, lawson}, groupoids \cite{groupoids-blends, duwenig-li}, and more general categories \cite{matched-pair} have also been developed in recent years. 
In \cite[Lemma 3.27]{matched-pair}, Mundey and Sims show  that every $(k_1 + k_2)$-graph $\Omega$  
decomposes as a Zappa--Sz\'ep product of a $k_1$-graph and a $k_2$-graph, $\Omega = \Omega_1 \bowtie \Omega_2$.   
However, their decomposition does not mesh neatly with the $C^*$-algebraic structure. 
For example, a product should be a special case of a Zappa--Sz\'ep product, where both $G$ and $H$ act trivially on each other, yet the Mundey--Sims factors $\Omega_1, \Omega_2$ of a product $k$-graph $\Omega = \Lambda \times \Gamma$ do not coincide with $\Lambda$ and $\Gamma$.  in 
In particular, $C^*(\Omega) \not\cong C^*(\Omega_1) \otimes C^*(\Omega_2)$. 

In this paper, we refine the Mundey--Sims approach to $k$-graph Zappa--Sz\'ep products using a graphical perspective on higher-rank graphs. Formally, a $k$-graph is a countable category $\Omega$ together with a degree functor $d: \Omega \to \N^k$ which satisfies a certain factorization property (see Definition \ref{def:kgraph} below). However (cf.~\cite{pqr, hrsw, efgggp}) one can equivalently  view a $k$-graph  as a quotient of an edge-colored directed graph with $k$ colors of edges. Indeed, if $\Omega = \Lambda \times \Gamma$ is a product $k$-graph, then the underlying directed graph of $\Omega$ is the box product, or Cartesian product, of the graphs of the factors: $G(\Omega) \cong G(\Lambda) \square G(\Gamma).$   (The converse is not true, however; see Example \ref{ex: bouq} below.) 

If $G(\Omega) \cong G(\Lambda\times\Gamma)$, then as we observe in Proposition \ref{dis-union-prop} below, the graphs $G(\Omega_1), G(\Omega_2)$ of the Mundey--Sims factors $\Omega_1, \Omega_2$ decompose as disjoint unions in a way which 
is compatible with the product  structure of $G(\Omega)$: 
\begin{equation} G(\Omega_1) \cong \bigsqcup_{x \in \Gamma^0} G(\Lambda) \times \{ x\}, \qquad  G(\Omega_2) \cong \bigsqcup_{v \in \Lambda^0} \{ v\} \times G(\Gamma). 
\label{eq:graphs}
\end{equation}
Because of this, when  a $k$-graph $\Omega$ satisfies $G(\Omega) \cong G(\Lambda\times\Gamma)$ (we call such $k$-graphs {\em quasi-products}), we can reinterpret the Zappa--Sz\'ep actions $\lact: \Omega_1 \to \Aut(\Omega_2), \ \ract: \Omega_2 \to \Aut(\Omega_1)$ from \cite{matched-pair} as ``actions'' of $G(\Lambda)$ and $G(\Gamma)$ on each other. 
In general, these ``actions'' barely deserve the name, as Examples \ref{ex:rho-non-comp} and \ref{ex:rho-non-isom} highlight. However, when $\Omega = \Lambda \times \Gamma$ is a product $k$-graph, both ``actions'' are trivial -- just as a product of groups can be viewed as the Zappa--Sz\'ep product of two groups acting trivially on each other.

Inspired by this observation, we set out to analyze the analogue, in this picture, of a semidirect product.  That is, what can we say if one of the ``actions'' (say, the action of $G(\Lambda)$ on $G(\Gamma)$) is trivial? 
 
 In this case, Proposition \ref{prop:stable-factorization} and Theorem \ref{thm: rho isom} establish that the decomposition \eqref{eq:graphs} holds at the $k$-graph level: $$\Omega_1 \cong \bigsqcup_{v\in \Gamma^0} \Lambda \times \{v\}, \ \Omega_2 \cong \bigsqcup_{x \in \Lambda^0} \{ x\} \times \Gamma.$$  Moreover, if $G(\Lambda)$ acts trivially on $G(\Gamma)$ and $C^*(\Gamma) \cong C^*(\Z^\ell)$ can also be viewed as a group $C^*$-algebra, we show in Theorem \ref{thm:circle-xprod} that the quasi-product $C^*$-algebra $C^*(\Omega)$ is a crossed product, $C^*(\Omega) \cong C^*(\Lambda) \rtimes \Z^\ell$.

Perhaps surprisingly,  Example \ref{ex: path loops} reveals that  a quasi-product $\Omega$ need not have both ``actions'' trivial in order to admit an isomorphism $\Omega \cong \Lambda \times \Gamma$ to a product graph.  However, Theorems \ref{thm:relaxed stable} and \ref{thm:product-isom} describe the precise conditions on the two ``actions''   which will guarantee a product decomposition $\Omega \cong \Lambda \times \Gamma$, and hence a tensor product decomposition  $C^*(\Omega) \cong C^*(\Lambda) \otimes C^*(\Gamma)$. That is, our results in Section \ref{sec:qp=is-product} give a graph-theoretic characterization of when a $k$-graph decomposes, as a category, into a product.

\section{Preliminaries}
\textbf{Notation:} We will suppose that $0\in\N$; by abuse of notation, we also write $0 \in \N^k$ for any $k$. We let $E:=\{e_1,\cdots, e_k\}\subseteq \N^k$ be the standard basis for $\N^k$. 

\subsection{Higher-rank graphs}
\begin{dfn}
\label{def:kgraph}
    \cite[Def 1.1]{kp}  Let $\Lambda$ be a countable category and $d:\Lambda \to \N^k$ a functor. If $(\Lambda, d)$ satisfies the \textit{factorization property}---that is, for every morphism $\lambda \in \Lambda$ and $n,m\in \N^k$ such that $d(\lambda)=m+n$, there exist unique $\mu,\nu\in \Lambda$ such that $d(\mu)=m$, $d(\nu)=n$, and $\lambda = \mu\nu$---then $(\Lambda,d)$ is a \textit{$k$-graph} (or higher-rank graph of rank $k$).  We will frequently abuse notation and refer to $\Lambda$ as a $k$-graph, without mentioning the functor $d$.
\end{dfn}

For a $k$-graph $\Lambda$ and $A\subseteq \N^k$ we define
\[\Lambda^A:=\{\lambda\in \Lambda : d(\lambda)\in A\}.\]
We suppress the set-builder notation for singleton sets. For example, $\Lambda^{e_1}:=\Lambda^{\{e_1\}}$.
By the factorization rule, for each $\lambda \in \Lambda$ there exist unique $v, w \in \Lambda^0$ so that $\lambda = v \lambda w$; we write $v=: r(\lambda)$ and $w=: s(\lambda)$.
Combining this with the notation above, we obtain the sets 
\[ v \Lambda^n = \{ \lambda \in \Lambda: r(\lambda) = v, d(\lambda) = n\}, \quad v\Lambda w = \{ \lambda \in \Lambda: r(\lambda) = v, s(\lambda) = w\} \]
and so on.

We say that $\Lambda$ is {\em row-finite} if $v\Lambda^n$ is finite for all $n \in \N^k$. 

\begin{dfn}
\label{1sk}
    For a higher-rank graph $\Lambda$, we define its \textit{one-skeleton} as an edge-colored directed graph 
    \[G(\Lambda):= (\Lambda^0,\Lambda^E,s,r,d),
    \]
where $s, r: \Lambda^E \to \Lambda^0$ are defined as above, 
  and we view $d(e) \in E$ as the color of the edge $e.$
We will say that $\Lambda$ is {\em connected} if the underlying undirected graph of $G(\Lambda)$ is connected; that is, for any $v, w \in \Lambda^0,$ there is a (not necessarily composable) string of edges $g_1, \ldots, g_n \in \Lambda^E$ linking $v,w$. Precisely, we require that for all $i$, $\{ s(g_i), r(g_i)\} \cap \{ s(g_{i+1}), r(g_{i+1})\} \not= \emptyset$, $w \in \{ s(g_n), r(g_n)\}$ and $v \in \{ s(g_1), r(g_1)\}.$
    \end{dfn}

Conversely, given a directed graph $G= (G^0, G^1, r, s)$,  its path category $G^*$ is
\[ G^* =  \bigcup_{n\in \N} \{ g_1 \cdots g_n \in (G^1)^n: s(g_i) = r(g_{i+1}) \text{ for all } 1 \leq i \leq n\}.  \]
If $G = G(\Lambda)$, then  the factorization rule implies that each $\lambda \in \Lambda$ has multiple representations in $G^*.$  For example, if $d(\lambda) = e_1 + e_2,$ the fact that $e_1 + e_2 = e_2 + e_1$ means that $\lambda$ can be written in two ways as a concatenation of edges: there are unique $e, e' \in \Lambda^{e_1}$ and $f, f'\in \Lambda^{e_2}$ such that
\[ \lambda =  e f = f'e'.\]
Consequently, $\Lambda \not= G(\Lambda)^*,$ but we can view $\Lambda$ as a quotient of the path category $G(\Lambda)^*$.  

However, not every equivalence relation on the path category $G^*$ of an edge-colored directed graph $G$ will yield a $k$-graph $G^*/\sim$.  Our next goal is to describe those  equivalence relations that will. (For consistency with the $k$-graph setting, we view an edge-coloring on a directed graph as a map $d: G^1 \to E = \{ e_1, \ldots, e_k\}$; the generators of $\N^k$ represent the allowed colors.)

\begin{dfn}
\label{def:color-order}
    The {\em color order} of a path $p = p_1 \cdots p_n$ in an edge-colored directed graph $G$ is the $n$-tuple $(d(p_1), d(p_2), \ldots, d(p_n)) \in E^n.$
\end{dfn}
By the factorization property, if $G^*/\sim$ is a $k$-graph and $\lambda \in G^*$,  each permutation of the color order of $\lambda$ is represented by a unique path $\mu$ with $\mu \sim \lambda$.
It turns out  that this condition indeed characterizes when $G^*/\sim$ is a $k$-graph: 

\begin{thm}\cite[Theorem 2.1]{efgggp}
\label{thm:kgraph}
  Let $G$ be an edge-colored directed graph. If   $\sim$ is any $(s,r, d)$-preserving equivalence relation $\sim$ on the path category $G^*$ of an edge-colored directed graph $G$, then $G^*/\sim$ is a $k$-graph iff $\sim$ satisfies   the conditions (KG0)  and (KG4)  below.
\end{thm}

\begin{itemize}
  \item[(KG0)] If $\lambda \in G^*$ is a path such that $\lambda = \lambda_2 \lambda_1$, then $[\lambda] = [p_2 p_1]$ whenever $p_1 \in [\lambda_1]$ and $p_2 \in [\lambda_2]$.

\item[(KG4)] For every $\lambda \in G^*$ and every permutation of the color order of $\lambda$, there is a unique path in $[\lambda]$ with the permuted color order.
\end{itemize}

An equivalence relation $\sim$ on $G^*$ which makes $G^*/\sim$ into a $k$-graph will be called a {\em factorization rule}.

\begin{example}
\label{ex:products}
    If $(\Lambda, d_1)$ is a $k_1$-graph and $(\Gamma, d_2)$ is a $k_2$-graph, then \cite[Proposition 1.8]{kp} establishes that the product category $(\Lambda \times \Gamma, d_1 \times d_2)$ is a ($k_1 + k_2)$-graph. Thus, $(\Lambda \times \Gamma)^0 = \Lambda^0 \times \Gamma^0$, and the source and range maps evaluate component wise. 
 Consequently, for any $\lambda  \in \Lambda, \gamma \in \Gamma$, both of the pairs $((\lambda, r(\gamma)), (s(\lambda), \gamma))$ and $((r(\lambda), \gamma),(\lambda, s(\gamma)))$ are composable in $\Lambda \times \Gamma$, and both pairs have the product $(\lambda, \gamma)$.
  
In fact, $G(\Lambda \times \Gamma)$ is the Cartesian product of $G(\Lambda)$ and $G(\Gamma)$: 
\[G(\Lambda \times \Gamma)^0 = \Lambda^0 \times \Gamma^0 = G(\Lambda)^0 \times G(\Gamma)^0 \quad \text{ and } \quad   G(\Lambda \times \Gamma)^1 = G(\Lambda)^1 \times \Gamma^0 \,  \sqcup \,  \Lambda^0 \times G(\Gamma)^1.\]
    
   To link this example with Theorem \ref{thm:kgraph}, we now describe $G(\Lambda\times\Gamma)^*$ and the factorization rule $\sim_\times$ yielding the $(k_1 + k_2)$-graph $\Lambda \times \Gamma$. For any $v \in \Gamma^0$ and any paths $p, q, p', q' \in G(\Lambda)^*$, we have  $(p, v)(q, v) \sim_\times  (q', v)(p', v)$ if $pq \sim_\Lambda q'p'$ in $\Lambda$; for any $x \in \Lambda^0$ and any paths $h, h', \ell, \ell' \in G(\Gamma)$, 
 $(x, h)(x, \ell) \sim_\times (x, \ell')(x, h')$ if $h \ell  \sim_\Gamma \ell' h'$  in $\Gamma$; and 
   \[ (p, r(h))(s(p), h) \sim_\times  ( r(p), h)(p, s(h) ).\]
   There are frequently other factorization rules on $G(\Lambda \times \Gamma)^*$, which will yield different $(k_1+k_2)$-graphs.
\end{example}

\begin{remark}
\label{rmk:commuting-squares}
    If $G$ is an edge-colored directed graph with 2 colors of edges, it follows from Theorem \ref{thm:kgraph} that  a factorization rule on $G$ is completely determined by its {\em commuting squares}, or the equivalence classes of length-2 two-colored paths in $G$.  However, if $G$ is a $k$-colored directed graph and $k > 2$, specifying the commuting squares may not determine a valid factorization rule, as we show in Example \ref{ex: bouq}.
\end{remark}

\begin{example}
\label{ex: bouq}
Consider an edge-colored bouquet graph with three black edges, $f_1$, $f_2$, and $f_3$, a single blue edge, $e$, and a single red edge, $g$.
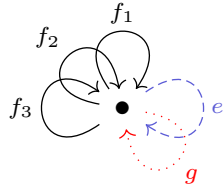
    \begin{center}
     \scalebox{1.25}{
        \begin{tikzcd}
	\bullet
	\arrow["{f_2}", from=1-1, to=1-1, loop, in=100, out=170, distance=10mm]
	\arrow["e", color={rgb,255:red,92;green,92;blue,214}, dashed, from=1-1, to=1-1, loop, in=325, out=35, distance=10mm]
	\arrow["{f_1}", from=1-1, to=1-1, loop, in=55, out=125, distance=10mm]
    \arrow["{f_3}", from=1-1, to=1-1, loop, in=145, out=215, distance=10mm]
    \arrow["{g}", from=1-1, to=1-1, loop, in=280, out=-10,dotted, color={rgb,255:red,255;green,0;blue,0}, distance=10mm]
\end{tikzcd}}
\end{center}

There are 36 different ways to specify commuting squares for the graph above; however, in all of them, we will have $eg \sim ge$ since there is only one red edge and only one blue edge. We pause to further examine three of these  which we name $\sim_1,\sim_2,$ and $\sim_3$. They are defined  as follows:

\begin{align*}
    f_1e &\sim_1 ef_1 & f_2 e &\sim_1 ef_2 & f_3 e &\sim_1 ef_3 & f_1g &\sim_1 gf_1 & f_2 g &\sim_1 gf_2 & f_3 g &\sim_1 gf_3\\
    f_1e &\sim_2 ef_2 & f_2 e &\sim_2 ef_1 & f_3 e &\sim_2 ef_3 & f_1g &\sim_2 gf_2 & f_2 g &\sim_2 gf_1 & f_3 g &\sim_2 gf_3\\
    f_1e &\sim_3 ef_1 & f_2 e &\sim_3 ef_3 & f_3 e &\sim_3 ef_2 & f_1g &\sim_3 gf_2 & f_2 g &\sim_3 gf_1 & f_3 g &\sim_3 gf_3
\end{align*}

First, we note that $\sim_3$ does not produce a $3$-graph. In particular, the path $f_1eg$ is equivalent to two paths of color order ``black, blue, red:" 
\[(f_1e)g\sim_3 (ef_1)g  = e (f_1 g)\sim_3 e(gf_2) = (eg) f_2\sim_3 (ge)f_2,\]
but also $ f_1(eg)\sim_3 (f_1 g)e \sim_3 g(f_2e)\sim_3gef_3.$  

The relations $\sim_1$ and $\sim_2$ do yield $3$-graphs which we distinguish as $\Omega_1$ and $\Omega_2$. In fact, $\Omega_1$ is a product $3$-graph: $\Omega_1 = B_3 \times C_{1,2}$ is the product of a bouquet $B_3$ with three loops, and the length-1 cycle in two colors, $C_{1,2}$ (cf.~Definition \ref{cycle-defn} below).
Therefore, \cite[Corollary 3.5(iv)]{kp} implies that $C^*(\Omega_1)\cong \mathcal O_3 \otimes C(\T^2)$. 

 On the other hand, to check that $\sim_2$ does indeed define a 3-graph, \cite[Theorem 4.4]{hrsw} implies that it is enough to show that (KG4) holds for all length-3 three-color paths.  This is a straightforward computation which we leave to the reader. At the $C^*$-algebra level, $C^*(\Omega_2)\cong \mathcal O_3 \rtimes_{\tau} \Z^2$ where $\tau_{e_1}$ and $\tau_{e_2}$ both interchange the generators $s_1$ and $s_2$ in $\mathcal{O}_3$. 
\end{example}

\subsection{Zappa--Sz\'ep products}
In group theory, there is a generalization of the semi-direct product known as the Zappa--Sz\'ep product, where the two factor groups each act on the  other in a compatible way. In recent years, this structural theory has been extended to groupoids and small categories \cite{groupoids-blends,duwenig-li, matched-pair}, and in particular to $k$-graphs \cite{matched-pair}, where these two-sided actions are called {\em matched pairs}.

Recall that, if $\mathcal C, \mathcal D$ are categories with the same object space $(\mathcal C^0 = \mathcal D^0$), then (using the arrows-only picture of category theory) we denote 
\[ \mathcal C * \mathcal D = \{ (c, d) \in \mathcal C \times \mathcal D : s(c) = r(d) \in \mathcal C^0 = \mathcal D^0\}.\]
Following \cite{matched-pair}, a {\em left action} of a category $\mathcal C$ on another category $\mathcal D$ with $\mathcal C^0 = \mathcal D^0$ is a map $\mathcal C * \mathcal D \to \mathcal D,$ denoted $c \lact d,$ such that $(c_1 c_2) \lact d = c_1 \lact (c_2 \lact d)$ whenever $s(c_1) = r(c_2).$
Similarly, a {\em right action} of $\mathcal D$ on $\mathcal C$ 
is a map $\mathcal C * \mathcal D \to \mathcal C,$ denoted $c \ract d,$ such that $c \ract (d_1 d_2) = (c \ract d_1) \ract d_2$ whenever $s(d_1) = r(d_2).$
\begin{dfn}
    \cite[Def. 3.1]{matched-pair} A \textit{matched pair} is a quadruple $(\calC,\calD,\lact,\ract)$ consisting of small categories $\calC$ and $\calD$ with $\calC^0=\calD^0$, a left action 
    $\lact: \calC * \calD \to \calD$, and a right action 
    $\ract: \calC * \calD \to \calC$ such that for all composable quadruples $(c_1,c_2,d_1,d_2)\in \calC^2 * \calD^2$,
    \begin{enumerate}[label=(MP\arabic*)]
        \item $s(c_2\lact d_1) = r(c_2\ract d_1)$, 
        \item $c_2\lact (d_1d_2) = (c_2\lact d_1)((c_2\ract d_1)\lact d_2)$, and 
        \item $(c_1c_2)\ract d_1 = (c_1 \ract (c_2\lact d_1))(c_2\ract d_1)$.
    \end{enumerate}
    \label{def:matched-pair}
\end{dfn}

\begin{lem}
    \label{k-graph-mp-lem} \cite[Lemma 3.27]{matched-pair} Every $(k_1+k_2)$-graph, $\Omega$, decomposes uniquely as a  matched pair $(\Omega_1,\Omega_2,\lact,\ract)$ where $\Omega_1:=d^{-1}(\N^{k_1}\times \{0\})$ is a $k_1$-graph and $\Omega_2:=d^{-1}(\{0\}\times \N^{k_2})$ is a $k_2$-graph such that    $\omega_1\omega_2=(\omega_1\lact\omega_2)(\omega_1\ract\omega_2)$ for all  $(\omega_1,\omega_2)\in \Omega_1 * \Omega_2$.  Furthermore, $d(\omega_1 \lact \omega_2) = d(\omega_2)$ and $d(\omega_1 \ract \omega_2) = d(\omega_1)$. 
\end{lem}

The factorization rule of $\Omega$ can thus be described in terms of $\lact, \ract$ as follows.
If $\omega_1, \omega_2$ are composable edges in $\Omega_1, \Omega_2$ respectively, the above Lemma implies that $\omega_1\lact \omega_2$ and $\omega_1\ract\omega_2$ are also edges, and moreover that 
$$\omega_1\omega_2\sim (\omega_1\lact \omega_2)(\omega_1\ract\omega_2).$$
Conversely (see \cite[Lemma 3.28]{matched-pair}), the factorization rule on $\Omega$ is completely determined by the factorization rules on $\Omega_1, \Omega_2$ and the Zappa--Sz\'ep actions on edges. We will take advantage of this perspective in the proof of Theorem \ref{thm:relaxed stable} below. 

For a concrete example, consider $\sim_2$ from Example \ref{ex: bouq}. Setting $\Omega_1$ to be the graph given by the black edges and $\Omega_2
$ the 2-graph generated by $e,g$,  
the associated Zappa--Sz\'ep actions $\lact_2$ and $\ract_2$ satisfy 
\[ f_i\lact_2 e = e, \;  f_i \lact_2 g = g, \;  f_1\ract_2 e=f_2 = f_1 \ract_2 g , \ \text{ and } \ f_3\ract_2 e= f_3 = f_3 \ract_2 g.\]

\section{Quasi-products}
\label{sec:stable}
A quasi-product (see Definition \ref{def:quasi-factor} below) is a $(k_1+k_2)$-graph which at the one-skeleton level resembles a product graph. In this section, we 
study the properties of the Zappa--Sz\'ep actions $\lact, \ract$ on quasi-products at the graph level. While these ``actions'' need not be particularly well-behaved  (cf.~Examples \ref{ex:rho-non-comp} and \ref{ex:rho-non-isom} below), if one of them is trivial, we obtain an elegant structural decomposition of the quasi-product $\Omega$ (Proposition \ref{prop:stable-factorization} and Theorem \ref{thm: rho isom}).

\begin{dfn}
\label{def:quasi-factor}
    Let $\Omega$ be a $(k_1+k_2)$-graph. If there exist $k_1$- and $k_2$-graphs, $\Lambda$ and $\Gamma$, such that 
    \begin{enumerate}[(i)]
        \item $G(\Omega)\cong G(\Lambda\times\Gamma)$ and
        \item there exist injective functors $\psi_1: \Lambda \into \Omega$ and $\psi_2:\Gamma \into \Omega$
        \item with $\pi_i: \N^{k_1+k_2}\to \N^{k_i}$ the canonical projection, $d_\Lambda=\pi_1d_\Omega\psi_1$ and $d_\Gamma=\pi_2 d_\Omega \psi_2$,
    \end{enumerate}
    then we call $\Omega$ a {\em quasi-product} with {\em quasi-factors} $\Lambda$ and $\Gamma$.  
\end{dfn}

\begin{remark}
\label{rmk:mp-not-quasi-prod}
   Let $\Omega$ be a $(k_1 + k_2)$-graph and consider its  matched-pair decomposition  $(\Omega_1, \Omega_2, \lact, \ract) $ from Lemma \ref{k-graph-mp-lem}.  Then $\Omega$ is not a quasi-product with quasi-factors $\Omega_1$ and $\Omega_2$, because 
   $\Omega_1^0 = \Omega_2^0 = \Omega^0.$
   Therefore, 
   \[ (\Omega_1 \times \Omega_2)^0 = \Omega^0 \times \Omega^0 \not= \Omega^0,\]
   so we do not have $G(\Omega) \cong G(\Omega_1 \times \Omega_2)$.
\end{remark}

When $\Omega$ is a quasi-product with quasi-factors $\Lambda$ and $\Gamma$, we will often
use the notation for edges and paths in $G(\Lambda\times\Gamma)$ (see Example \ref{ex:products}) to describe edges in $G(\Omega)$. This emphasizes the fact that the only difference between $\Lambda\times\Gamma$ and $\Omega$ is their factorization rules, $\sim_\times$ and $\sim_\Omega$, respectively.

\begin{prop}
    \label{dis-union-prop}
   If $\Omega$ is a quasi-product $(k_1+k_2)$-graph with associated matched pair $(\Omega_1, \Omega_2, \lact, \ract)$ as in Lemma \ref{k-graph-mp-lem}, then $\Omega_1$ and $\Omega_2$ both decompose as disjoint unions:
\[ \Omega_1 = \bigsqcup_{w\in \Gamma^0} \Omega_1^w\times\{w\}  \quad \text{ and } \quad \Omega_2 =  \bigsqcup_{x\in \Lambda^0} \{x\}\times\Omega_2^x.\]
\end{prop}

\begin{proof}
   Note that $\Omega_1, \Omega_2$ have one-skeletons $G(\Omega_1)$ and $G(\Omega_2)$. To investigate the structure of these one-skeletons, we split the standard basis $E:=\{e_1,\cdots, e_{k_1+k_2}\}$ of $\N^{k_1+k_2}$ into $E_1:=\{e_1,\cdots, e_{k_1}\}$ and $E_2=E\setminus E_1$. 
Then 
    \[G(\Omega_1) = (\Omega^0, \Omega^{E_1}, s, r) \quad \text{ and } \quad G(\Omega_2) = (\Omega^0,\Omega^{E_2},s,r).\]

As $G(\Omega) \cong G(\Lambda\times \Gamma)$, every edge in $G(\Omega)$ is of the form $(e, v)$ or $(x,f)$ for some $e\in G(\Lambda), f\in G(\Gamma), x \in \Lambda^0, v \in \Gamma^0. $
     
    Since $s$ and $r$ evaluate coordinate-wise, if $v\neq w \in \Gamma^0$, then $G(\Lambda)\times \{v\}$ is disconnected from $G(\Lambda)\times\{w\}$ in $G(\Omega_1)$. A similar argument applies to $G(\Omega_2)$ giving the simplification

    \[G(\Omega_1)  = \bigsqcup_{w\in \Gamma^0} G(\Lambda)\times \{ w \} \quad \text{and} \quad G(\Omega_2) = \bigsqcup_{x\in \Lambda^0} \{x\}\times G(\Gamma).\]

The factorization rule of $\Omega$ restricts to connected components; that is, for any $w \in \Gamma^0, x \in \Lambda^0$ we have $k_i$-graphs
\[ (G(\Lambda)\times \{w\} )/\sim_\Omega =: \Omega_1^w \times\{w\}
, \qquad (\{x\}\times G(\Gamma) )/\sim_\Omega =:  \{x\}\times
\Omega_2^x \]    
which yields the desired decomposition of $\Omega_1, \Omega_2$.
\end{proof}

In any $k$-graph $\Omega$ with associated matched pair $(\Omega_1, \Omega_2, \lact, \ract)$, the fact that (Lemma \ref{k-graph-mp-lem}) $\omega_1 \omega_2 = (\omega_1 \lact \omega_2) (\omega_1 \ract \omega_2)$ whenever $(\omega_1, \omega_2) \in \Omega_1 * \Omega_2$ yields 
\begin{equation}
\label{eq:mp-source-range}
s(\omega_1 \ract \omega_2) = s(\omega_2) \qquad \text{ and } \qquad r(\omega_1 \lact \omega_2) = r(\omega_1).
\end{equation}
If $\Omega$ is a quasi-product, then for any $p \in G(\Lambda)^*, q \in G(\Gamma)^*$, if we set 
\[ \omega_1 = ([p], r(q)) \in \Omega_1^{r(q)} \times \{r(q)\}, \qquad \omega_2 = (s(p), [q]) \in \{s(p)\} \times  \Omega_2^{s(p)},\]
then $(\omega_1,\omega_2)\in \Omega_1*\Omega_2$ is a composable pair.  Since $d(\omega_1 \ract \omega_2) = d(\omega_1)$ and $d(\omega_1 \lact \omega_2) = d(\omega_2),$ we therefore have $\omega_1 \ract \omega_2 \in \Omega_1^{s(q)}\times\{s(q)\}$ and $\omega_1 \lact \omega_2 \in \{r(p)\}\times\Omega_2^{r(p)}$.

As $G(\Omega_1^w) = G(\Lambda)$ for all $w\in \Gamma^0$ and $G(\Omega_2^x) = G(\Gamma)$ for all $x \in \Lambda^0$, we can obtain, from the Zappa--Sz\'ep structure and the decomposition above, an 
action of paths in $G(\Lambda)$ on paths in $G(\Gamma)$ and vice versa.

\begin{dfn} 
Let $\Omega$ be a quasi-product with quasi-factors $\Lambda, \Gamma$.  For any $p \in G(\Lambda)^*, q\in G(\Gamma)^*$, we define color order preserving functions $p_\lact:G(\Gamma)^*\to G(\Gamma)^*$ and $q_\ract:G(\Lambda)^*\to G(\Lambda)^*$, via 
\label{def:path-action}
\begin{align}
    (r(p), [p_\lact (q)])   &: = ([p], r(q)) \lact (s(p), [q])\\ 
  ([ q_\ract(p) ], s(q))  &: = ([p], r(q)) \ract (s(p), [q]). 
\label{eq:path-action}
\end{align} 
\end{dfn}

\begin{prop}
   Let $\Omega$ be a quasi-product with quasi-factors $\Lambda$ and $\Gamma$. For any $p \in G(\Lambda)^*, q \in G(\Gamma)^*$, 
   \[ r(q_\ract (p)) = r(p), \ s(q_\ract (p)) = s(p) \quad \text{ and } \quad r(p_\lact (q)) = r(q), \ s(p_\lact (q)) = s(q). \]

 \label{prop:source-range-general}
\end{prop}
\begin{proof}
Equation  \eqref{eq:mp-source-range} implies that $r(r(p), [p_\lact (q)])= r([p], r(q)) = (r(p), r(q))$ and similarly that $s([q_\ract (p)], s(q)) = s(s(p), [q]) = (s(p), s(q)).$  Conseqently, 
\[ r(p_\lact (q)) = r(q) \in \Gamma^0 \qquad \text{ and } \qquad s(q_\ract (p)) = s(p) \in \Lambda^0.\]
For the remaining assertions, recall that, thanks to (MP1), $$((r(p),[p_\lact (q)]), ([q_\ract (p)], s(q))) \in \Omega_2 * \Omega_1$$ is a composable pair. In particular, $s(r(p), [p_\lact (q)]) = (r(p), s(p_\lact (q)))$ must equal $r([q_\ract (p)], s(q)) = (r(q_\ract (p)), s(q))$. The conclusion follows. 
    
\end{proof}

In general, $p_\lact$ and $q_\ract$ need not be particularly well behaved. Although they are degree-preserving by Lemma \ref{k-graph-mp-lem} and source- and range-preserving by Proposition \ref{prop:source-range-general}, one or both may fail to be injective or surjective (see Example \ref{ex:rho-non-isom}), or
 may not respect composition. 

 \begin{example}
\label{ex:rho-non-comp}
Suppose that $\Omega$ is a quasi-product $2$-graph (one-skeleton pictured to the right) with quasi-factors $\Lambda$ and $\Gamma$ (digraphs pictured to the left).
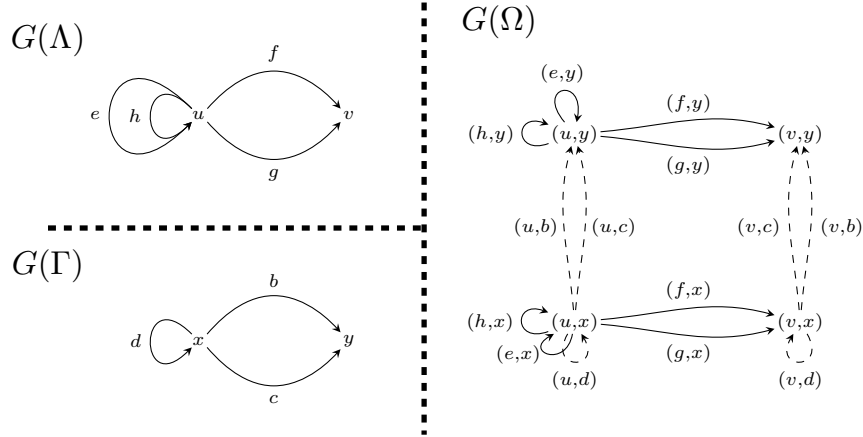
\[ 
\begin{tikzpicture}[ >=stealth]


\draw[-,dashed,line width=2] (3,1.5) to[](3,-4.25);
\draw[-,dashed, line width=2] (-2,-1.5) to[] (3,-1.5);

\node[scale=1.25] at (-2,1) {$G(\Lambda)$};
\node[scale=1.25] at (-2,-2) {$G(\Gamma)$};
\node[scale=1.25] at (4,1.25) {$G(\Omega)$};


\node[inner sep=0.8pt] (u) at (0,0) {$\scriptstyle u$}
edge[loop,->,in=225,out=135,looseness=15] node[auto,black,swap] {$\scriptstyle h$} (u)
edge[loop,->,in=225,out=135,looseness=30] node[auto,black, left] {$\scriptstyle e$} (u);

\node[inner sep=0.8pt] (v) at (2,0) {$\scriptstyle v$}
;

\draw[->] (u.north east)
    parabola[parabola height=0.5cm] (v.north west);

\node[inner sep=6pt, above] at (1,0.5) {$\scriptstyle f$};
\draw[->] (u.south east)
    parabola[parabola height=-0.5cm] (v.south west);

\node[inner sep=6pt, below] at (1,-0.5) {$\scriptstyle g$};

\node[inner sep=0.8pt] (x) at (0,-3) {$\scriptstyle x$}
edge[loop,->,in=225,out=135,looseness=15] node[auto,black,swap] {$\scriptstyle d$} (x);

\node[inner sep=0.8pt] (y) at (2,-3) {$\scriptstyle y$};

\draw[->] (x.north east)
    parabola[parabola height=0.5cm] (y.north west);

\node[inner sep=6pt, above] at (1,-2.5) {$\scriptstyle b$};
\draw[->] (x.south east)
    parabola[parabola height=-0.5cm] (y.south west);

\node[inner sep=6pt, below] at (1,-3.5) {$\scriptstyle c$};


\node[inner sep=0.8pt] (u) at (5,-2.75) {$\scriptstyle (u,x)$};

\draw[->] (u) to[in=160, out=-160, looseness=5] node[inner sep=3pt, left]{$\scriptstyle (h,x)$}  (u);

\node[inner sep=0.8pt] (v) at (8,-2.75) {$\scriptstyle (v,x)$};

\draw[->] (u) to[in=210, out=260, looseness=5] node[inner sep=3pt, left]{$\scriptstyle (e,x)$}  (u);

\draw[->] (u) to[out=5, in=165] node[inner sep=3pt, above]{$\scriptstyle (f,x)$} (v);

\draw[->] (u) to[out=-5, in=195] node[inner sep=3pt, below]{$\scriptstyle (g,x)$} (v);

\node[inner sep=0.8pt] (u2) at (5,-0.25) {$\scriptstyle (u,y)$};

\draw[->] (u2) to[in=160, out=-160, looseness=5] node[inner sep=3pt, left]{$\scriptstyle (h,y)$}  (u2);

\node[inner sep=0.8pt] (v2) at (8,-0.25) {$\scriptstyle (v,y)$};

\draw[->] (u2) to[in=80, out=130, looseness=10] node[inner sep=3pt, above]{$\scriptstyle (e,y)$}  (u2);

\draw[->] (u2) to[out=5, in=165] node[inner sep=3pt, above]{$\scriptstyle (f,y)$} (v2);

\draw[->] (u2) to[out=-5, in=195] node[inner sep=3pt, below]{$\scriptstyle (g,y)$} (v2);

\draw[dashed, ->] (u) to[out=85, in=285] node[inner sep=10pt,black,left] {$\scriptstyle (u,b)$} (u2);

\draw[dashed, ->] (u) to[out=95, in=255] node[inner sep=10pt,black,right] {$\scriptstyle (u,c)$} (u2);

\draw[dashed, ->] (v) to[out=85, in=285] node[black,right] {$\scriptstyle (v,b)$}  (v2); 

\draw[dashed, ->] (v) to[out=95, in=255] node[black,left] {$\scriptstyle (v,c)$}  (v2);

\draw[dashed,->] (u) to[in=300, out=240, looseness=8] node[ below]{$\scriptstyle (u,d)$}  (u);
\draw[dashed,->] (v) to[in=240, out=300, looseness=8] node[ below]{$\scriptstyle (v,d)$}  (v);

\end{tikzpicture}
\]

Define the commuting squares in $\Omega$ by 
\begin{align*}
    (f,y) (u,b) &= (v,c) (g,x);  
   & (f,y) (u,c) &= (v,b) (f,x);\\
    (h,y)(u,b) &= (u,c)(e,x); & (e,y)(u,b) &= (u,c)(h,x); \\
    (h,y)(u,c) &= (u,b)(h,x); & (e,y)(u,c) &= (u, b)(e,x);
\end{align*}
and in all other situations we use the product graph relations.  (For instance, $(f,x)(u,d) = (v,d)(f,x) $.)
Then (MP3) implies that 
\begin{align*}
    (b_\ract(fh),x)&  = \left( (f,y)(h,y) \right) \ract (u,b) = \left((f,y) \ract \left( (h,y) \lact (u,b) \right) \right)  (h, y) \ract (u,b) \\
    &= ((f,y) \ract (u,c) )(b_\ract (h),x) = 
    (c_\ract(f)b_\ract(h),x)  \not= (b_\ract(f) b_\ract(h), x)
\end{align*}
since $b_\ract(f) = g$ while $c_\ract(f) = f.$  That is, $b_\ract$ does not respect composition.
\end{example}

The potential failure of $p_\lact, q_\ract$ to respect composition arises from
(MP2) and (MP3). 
However, these axioms do give us formulas describing how $p_\lact, q_\ract$ behave on compositions of paths.
\begin{prop}
\label{prop:lact-composition}
    Let $\Omega$ be a quasi-product with quasi-factors $\Lambda$ and $\Gamma$.  If $p = p^1 p^2 \in G(\Lambda)^*, q = q^1 q^2 \in G(\Gamma)^*$, then 
    \[ p_\lact(q^1 q^2) = (p_\lact(q^1))(q^1_\ract(p))_\lact(q_2) \qquad \text{ and } \qquad q_\ract(p^1 p^2) = (p^2_\lact(q))_\ract(p^1)(q_\ract(p^2)). \]
\end{prop}
\begin{proof}
    By (MP2) and the fact that $r(q^1_\ract(p)) = r(p),$
    \begin{align*}
        ([p], r(q))\lact(s(p), [q^1 q^2]) &= ([p], r(q)) \lact \left( s(p), [q^1])(s(p), [q^2])\right) \\
        &= (r(p), [p_\lact(q^1)]) \left( ([q^1_\ract(p)], s(q^1)) \lact (s(p), [q^2])\right) \\
        &= (r(p), [p_\lact(q^1)]) ( r(p), [(q^1_\ract(p))_\lact(q^2)]).
    \end{align*}
    A symmetric argument, using (MP3), establishes the second claim.
\end{proof}

If $\Omega = \Lambda \times \Gamma$ is a product $(k_1+k_2)$-graph, then $p_\lact, q_\ract$ are much better behaved.  In that case, 
for any $\lambda \in \Lambda, \gamma \in \Gamma$,
\[ (\lambda, \gamma) = (\lambda, r(\gamma)) (s(\lambda), \gamma) = (r(\lambda), \gamma)(\lambda, s(\gamma)).\]

In other words, we have $p_\lact=\id$ and $q_\ract  = \id$ for all $p \in G(\Lambda)^*, q \in G(\Gamma)^*.$ In what follows, we show that assuming one of the equalities above to hold in a quasi-product $k$-graph leads to a crossed-product-like structure. 

\begin{dfn}
\label{stable-fact-dfn}
    For a quasi-product $\Omega$ with quasi-factors $\Lambda$ and $\Gamma$ we say that $\Gamma$ is a \textit{stable} quasi-factor if $p_\lact=\id$ for all $p\in G(\Lambda)^*$.
\end{dfn}

\begin{remark}
    \label{rmk:right-stability}
As an alternative to Definition \ref{stable-fact-dfn}, if $q_\ract = \id$ for all $p \in G(\Lambda)^*$, we could say that $\Lambda$ is a stable quasi-factor.  In this article, for simplicity and coherence of notation, we will phrase all results using $\Gamma$ as the stable quasi-factor, leaving  the straightforward translation to left-stability to the reader.
\end{remark}

The following Proposition shows that one can determine stability by checking the way $\lact$ behaves on edges.

\begin{prop}
\label{prop:edges-enough}
    Let $\Omega$ be a quasi-product with quasi-factors $\Lambda$ and $\Gamma$. The quasi-factor $\Gamma$ is stable iff $e_\lact(g) = g$ for all edges $e \in G(\Lambda), g \in G(\Gamma).$ 
\end{prop}
\begin{proof}
    Necessity is evident, so we prove sufficiency. Observe first that as $\lact$ is a left action, for any paths $p = p^1 p^2 \in G(\Lambda)^*, q \in G(\Gamma)^*$ we have 
    \begin{align*} ([p^1 p^2], r(q)) & \lact (s(p), [q]) = \left( ([p^1], r(q))([p^2], r(q))\right) \lact (s(p), [q]) \\
    & = ([p^1] , r(q)) \lact \left(([p^2], r(q)) \lact (s(p), [q])\right). \end{align*}
If $q, p^1, p^2$ are edges and we assume $e_\lact(g) = g$ for all edges $e \in G(\Lambda)^1, g \in G(\Gamma)^1$, the above becomes 
\[ ([p^1 p^2], r(q)) \lact (s(p), q) =(p^1, r(q)) \lact (r(p^2), q) = (r(p^1), q);\]
that is, $(p^1 p^2)_\lact(q) = q.$ 
By induction, it follows that $p_\lact(q) = q$ for all paths $p \in G(\Lambda)^*$ and all edges $q \in G(\Gamma)^1.$

 We now show that if $q = q^1 q^2$ is a path of length 2 in $G(\Gamma)^*$, then $p_\lact(q) = q$ for any $p \in G(\Lambda)^*$. By the previous paragraph,  $q^1_\ract(p)= :\tilde p \in G(\Lambda)^*$ also satisfies $\tilde p_\lact(g) = g$ for all edges $g \in G(\Gamma)^*.$ Thus,
  Proposition \ref{prop:lact-composition} implies that
    \begin{align*}
        p_\lact(q) &= (p_\lact(q^1))(q^1_\ract(p))_\lact(q^2)\\
        &= q^1 (\tilde p_\lact(q^2) )= q^1 q^2.
    \end{align*}
 By induction, we conclude that if $e_\lact(g) = g$ for all edges $e \in G(\Lambda)^1, g \in G(\Gamma)^1$, then for any paths $p \in G(\Lambda)^*, q \in G(\Gamma)^*$, we have $p_\lact(q) = q$; that is, $\Gamma$ is a stable quasi-factor.
\end{proof}

Definition \ref{stable-fact-dfn} imposes a substantial (perhaps a surprising) amount of structure on a stable quasi-product $k$-graph; the next several results detail this structure.  
To begin,  Definition \ref{stable-fact-dfn} is stated in terms of paths in $G(\Lambda)^*, G(\Gamma)^*;$ recall from Definition \ref{def:path-action} that we will interpret $p_\lact (q) \in G(\Gamma)^*$ as representing $([p], s(q)) \lact (s(p), [q]) \in \Omega_2^{r(p)},$ whereas we compute $[q]$ with respect to the equivalence relation in $\Omega_2^{s(p)}.$

However, the fact that for any $\omega \in \Omega_1, \omega \lact \cdot $ is a well-defined map on the $k_2$-graph $\Omega_2$ means that the  factorization rules on $G(\Gamma)^*$ yielding $\Omega_2^{s(p)}$ and $\Omega_2^{r(p)}$  agree in a stable quasi-product, as we now show.  That is, despite its  graphical definition, stability is a $k$-graph concept.

\begin{prop}
\label{prop: iso-comps-sigma}
Let $\Omega$ be a quasi-product with quasi-factor $\Lambda$ and stable quasi-factor $\Gamma$. For any $p \in G(\Lambda)^*$ and $q \in G(\Gamma)^*$ we have $[q]_{s(p)} = [q]_{r(p)}$.
In particular, if $\Lambda$ is connected, then for every $v , w\in \Lambda^0$ the factorization rules in $\Omega_2^v$ and $\Omega_2^w$ agree.
\label{prop:stable-factorization}
\end{prop}
\begin{proof}
	Fix a path $p \in G(\Lambda)$ and suppose $q, q' \in G(\Gamma)^*$ satisfy $q \sim q'$ in  $\Omega_2^{s(p)}$.  
 Since $\lact$ is defined at the level of the category, i.e.~the equivalence class, 
	\[ 
    (r(p), [p_\lact(q)]_{r(p)}) = ([p], r(q)) \lact (s(p), [q]_{s(p)})=([p], r(q')) \lact (s(p), [q']_{s(p)})
    \]
    must equal $ (r(p), [p_\lact(q')]_{r(p)}).$  
    Thus, if $\Gamma$ is stable, we must have $q \sim q'$ in $\Omega_2^{r(p)}$. 
    
    By definition, $p_\lact(q')$ has the same color order as $q'$.
As $q' \in [q]_{s(p)}$ was arbitrary,  and $[q]_{r(p)}$ (like $[q]_{s(p)})$ contains exactly one path of each color order,  this implies that the factorization rule in $\Omega_2^{r(p)}$ agrees with that in $\Omega_2^{s(p)}$.  From Definition \ref{def:quasi-factor}, we know there is $v \in \Lambda^0$ with $\Omega_2^v\cong \Gamma$ as $k_2$-graphs; the conclusion follows.
 \end{proof}

In particular, if $\Gamma$ is a stable quasi-factor and $\Lambda$ is connected, then for each $p \in G(\Lambda)^*,  p_\lact $
is not just a range-, source-, and degree-preserving map, but yields an isomorphism of $k_2$-graphs $ \Omega_2^{s(p)} \to \Omega_2^{r(p)}$, which we will denote $\widehat{p}_\lact$.  Theorem \ref{thm: rho isom} below shows that the stability of $\Gamma$ implies that $q_\ract$ also implements an isomorphism $\widehat{q}_\ract: \Omega_1^{r(q)} \to \Omega_1^{s(q)}$  of $k_1$-graphs for all $q\in G(\Gamma)^*$.
In certain cases, this reveals a crossed-product structure to $C^*(\Omega)$, as we discuss in Section \ref{sec:cycle-factor}.

\begin{thm}
\label{thm: rho isom}
    If $\Omega$ is a row-finite quasi-product with quasi-factor $\Lambda$ and stable quasi-factor $\Gamma$,  then for any $q \in G(\Gamma)^*$, $q_\ract$ induces a $k_1$-graph isomorphism $\widehat{q}_\ract: \Omega_1^{r(q)} \to \Omega_1^{s(q)}$. 
    If $\Gamma$ is connected, then for all $v, w \in \Gamma^0$ we have $\Omega_1^v \cong \Omega_1^w \cong \Lambda.$
\end{thm}
\begin{proof}
   We will check that the map $\widehat{q}_\ract: \Omega_1^{r(q)} \to \Omega_1^{s(q)}$ given by 
   \[ \widehat{q}_\ract([p]_{r(q)}) =  [q_\ract(p)]_{s(p)}
   \]
   is injective, surjective, and respects composition.

For injectivity, suppose that $[q_\ract(p)] = [q_\ract(p')]$.   Then $r(p) = r(p')$ and $s(p) = s(p')$, since $r(p) = r(q_\ract(p)) = r(q_\ract(p')) = r(p')$ by   Proposition \ref{prop:source-range-general}.  
Stability implies  that  
$  p_\lact(q) = q = p'_\lact( q) $, so 
 Lemma \ref{k-graph-mp-lem} yields 
   \begin{align*} 
   ([p], r(q))(s(p), [q]) & =  \left( ([p], r(q)) \lact (s(p), [q]) \right) \left( ([p], r(q)) \ract (s(p), [q]) \right) \\
   &= (r(p), [p_\lact(q)])([q_{\ract}(p)], s(q)) = (r(p), [p'_\lact(q)])([q_\ract(p')], s(q))  \\
   &= ([p'], r(q))(s(p'), [q]).
   \end{align*}
 Since $q_\ract(\cdot) $ preserves color order, our hypothesis that $[q_\ract(p)] = [q_\ract(p')]$ implies that  $d(p') = d(p)$.  
 The factorization property thus  implies that $([p'], r(q)) = ([p], r(q))$, proving injectivity.

   The stability of $\Gamma$, together with (MP3), implies that $\widehat{q}_\ract$ preserves composition:
   \begin{align*}
      ([q_\ract(p^1p^2)], r(q)) &= ([p^1p^2], r(q)) \ract (s(p^2), [q])  \\
       &= (([p^1], r(q))\ract (s(p^1),
       [p^2_\lact(q)]))([q_\ract(p^2)],r(q)) \\
       &= ([p^1],r(q))\ract (s(p^1),[q])([q_\ract(p^2)],r(q)) \\
       &=([q_\ract(p^1) q_\ract(p^2)],r(q))
   \end{align*}
   since $s(p^1) = r(p^2).$ 

   To see that $\widehat{q}_\ract$ is onto, recall that $\widehat{q}_\ract$ is well defined and injective, and $G(\Omega_1^{r(q)}) \cong G(\Omega_1^{s(q)}).$ 
   Moreover, since $\Omega$ is row-finite, there are only finitely many elements of $(v, r(q))\Omega_1^{r(q)} (w, r(q))$ of a given degree, for any $v, w \in \Lambda^0$.
 Since $\widehat{q}_\ract$ preserves source, range, and degree, the injectivity of $\widehat{q}_\ract$ forces it to be an isomorphism.
\end{proof}

The existence of the $k_1$-graph isomorphism $\widehat{q}_\ract$ is dependent on the stability of $\Gamma$ as a quasi-factor.  Perhaps surprisingly, it is not enough to assume that $\widehat{p}_\lact$ is a $k_2$-graph isomorphism in order to conclude that $\widehat{q}_\ract$ is a $k_1$-graph isomorphism.

\begin{example}
\label{ex:rho-non-isom}
Let $\Lambda $ be the following 1-graph with two vertices and four edges:
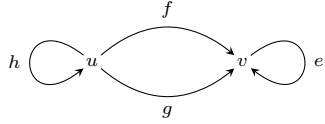
\[ 
\begin{tikzpicture}[yscale=0.75, >=stealth]

\node[inner sep=0.8pt] (u) at (0,0) {$\scriptstyle u$}
edge[loop,->,in=225,out=135,looseness=15] node[auto,black,swap] {$\scriptstyle h$} (u);

\node[inner sep=0.8pt] (v) at (2,0) {$\scriptstyle v$}
edge[loop,->,in=-45,out=45,looseness=15] node[auto,black,right] {$\scriptstyle e$} (v);

\draw[->] (u.north east)
    parabola[parabola height=0.5cm] (v.north west);

\node[inner sep=6pt, above] at (1,0.5) {$\scriptstyle f$};
\draw[->] (u.south east)
    parabola[parabola height=-0.5cm] (v.south west);

\node[inner sep=6pt, below] at (1,-0.5) {$\scriptstyle g$};
\end{tikzpicture}
\]

Let $\Gamma = \Lambda$; to reduce confusion, we will label the edges in $\Gamma$ $a,b,c,d$ instead of $e, f, g, h$ respectively, with vertices $x = r(d)$ and $y = r(a)$. Then the 1-skeleton of the quasi-product is 
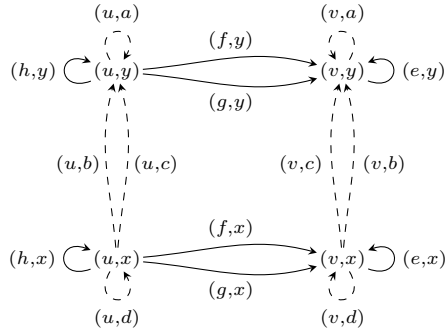
\[ 
\begin{tikzpicture}[ >=stealth]

\node[inner sep=0.8pt] (u) at (0,0) {$\scriptstyle (u,x)$};

\draw[->] (u) to[in=160, out=-160, looseness=5] node[inner sep=3pt, left]{$\scriptstyle (h,x)$}  (u);

\node[inner sep=0.8pt] (v) at (3,0) {$\scriptstyle (v,x)$};

\draw[->] (v) to[in=20, out=-20, looseness=5] node[inner sep=3pt, right]{$\scriptstyle (e,x)$}  (v);

\draw[->] (u) to[out=5, in=165] node[inner sep=3pt, above]{$\scriptstyle (f,x)$} (v);

\draw[->] (u) to[out=-5, in=195] node[inner sep=3pt, below]{$\scriptstyle (g,x)$} (v);

\node[inner sep=0.8pt] (u2) at (0,2.5) {$\scriptstyle (u,y)$};

\draw[->] (u2) to[in=160, out=-160, looseness=5] node[inner sep=3pt, left]{$\scriptstyle (h,y)$}  (u2);

\node[inner sep=0.8pt] (v2) at (3,2.5) {$\scriptstyle (v,y)$};

\draw[->] (v2) to[in=20, out=-20, looseness=5] node[inner sep=3pt, right]{$\scriptstyle (e,y)$}  (v2);

\draw[->] (u2) to[out=5, in=165] node[inner sep=3pt, above]{$\scriptstyle (f,y)$} (v2);

\draw[->] (u2) to[out=-5, in=195] node[inner sep=3pt, below]{$\scriptstyle (g,y)$} (v2);

\draw[dashed, ->] (u) to[out=85, in=285] node[inner sep=10pt,black,left] {$\scriptstyle (u,b)$} (u2);

\draw[dashed, ->] (u) to[out=95, in=255] node[inner sep=10pt,black,right] {$\scriptstyle (u,c)$} (u2);

\draw[dashed, ->] (v) to[out=85, in=285] node[black,right] {$\scriptstyle (v,b)$}  (v2); 

\draw[dashed, ->] (v) to[out=95, in=255] node[black,left] {$\scriptstyle (v,c)$}  (v2);

\draw[dashed,->] (u2) to[in=60, out=120, looseness=8] node[above]{$\scriptstyle (u,a)$}  (u2);
\draw[dashed,->] (v2) to[out=60, in=120, looseness=8] node[above]{$\scriptstyle (v,a)$}  (v2);
\draw[dashed,->] (u) to[in=300, out=240, looseness=8] node[ below]{$\scriptstyle (u,d)$}  (u);
\draw[dashed,->] (v) to[in=240, out=300, looseness=8] node[ below]{$\scriptstyle (v,d)$}  (v);

\end{tikzpicture}
\]

Define the commuting squares by 
\begin{align*}
    (g,y) (u,b) &= (v,b)(g,x); & 
   (f,y) (u,b) &= (v,c) (g,x); \\
    (g,y) (u,c) &= (v,c) (f,x); & (f,y) (u,c) &= (v,b) (f,x);
\end{align*}
and in all other situations we use the product graph relations.  (For instance, $(v,a)(f,y) = (f,y)(u,a)$.)
Then 
$(b_\ract(g),y) = (g,x) = (b_\ract(f),y)$, so $\widehat{b}_\ract$ is not an isomorphism. 

However, we claim that $\widehat{p}_\lact$ is an isomorphism for all paths $p \in G(\Lambda)^*.$ 
 Indeed, if $p$ contains $f$, then $p_\lact(q)$ swaps $b$ and $c$ if either of them occurs in $q$; otherwise, $p_\lact = id$.

To see this, we first consider a path $p$ of the form $e^n f h^m$. 
Observe  that
\[ (e, x)^n (f,x) (h,x)^m (u,d)^k = (e,x)^n (v,d)^k (f,x) (h,x)^m = (v,d)^k (e,x)^n (f,x) (h,x)^k,\]
and similarly $(e, y)^n (f,y) (h,y)^m (u,a)^k = (v,a)^k(e, y)^n (f,y) (h,y)^m .$  That is, on both $a^k$ and $d^k$, $(e^n f h^m)_\lact$ acts as the identity.
It remains to compute the action of $(e^n f h^m)_\lact$ on  paths of the form $a^j b d^k.$  Observe that
\begin{align*}
    (e,y)^n (f,y) (h,y)^m (u, a)^j (u,b) (u, d)^k & = (e,y)^n (f,y) (u,a)^j (h,y)^m (u,b) (u,d)^k \\
    & = (e,y)^n (v,a)^j (f,y) (u,b) (h,x)^m(u,d)^k \\
   & = (v,a)^j (e,y)^n (v,c) (g,x) (u,d)^k (h,x)^m \\
   &= (v,a)^j  (v,c) (e,x)^n (v,d)^k (g,x) (h,x)^m \\
    &= (v,a)^j (v,c) (v,d)^k (e,x)^n (g,x) (h,x)^m.
\end{align*}

Consequently, $(e^n f h^m)_\lact(a^j b d^k) = a^j c d^k.$
Similarly, $(e^n f h^m)_\lact(a^j c d^k) = a^j b d^k$. 

Similar computations to the above reveal that $(e^n g h^m)_\lact =id$. 
In particular, for each path $p$ in $G(\Lambda)^*$ and $q$ in $G(\Gamma)^*$, $p_\lact(q)$ contains the same number of edges with a given source and range as $q$ does, and if $p$ contains $f$ then $p_\lact$ interchanges $b$ and $c$ if either of them occurs in $q$.

We conclude that $p_\lact$ is a bijection (and, being range- and source-preserving, is thus an isomorphism) for all paths $p$ in $G(\Lambda)^*$.
\end{example}

Given the structural constraints imposed by stability, we conclude this section with the following method of guaranteeing that a $k_2$-graph is a stable quasi-factor.

\begin{dfn}
    Let $G$ and $F$ be edge-colored digraphs in $k$ colors. We say $\phi: G\to F$ is  a {\em homomorphism} of colored graphs if it preserves edge coloring and intertwines with $s$ and $r$. When $G=F$ and $\phi$ is bijective, we call $\phi$ an {\em automorphism}.  
\end{dfn}

\begin{lem}
\label{lem: stab-guarantee}
    Let $\mathcal F\subseteq \Aut(G(\Gamma))$ be the subgroup fixing vertices. If $\mathcal F=\{\id\}$, then $\Gamma$ is always a stable quasi-factor.
\end{lem}

\begin{proof}
    Let $p\in G(\Lambda)^*$ be arbitrary. As $p_\lact: G(\Gamma)^* \to G(\Gamma)^*$ 
   is range-, source- and degree-preserving, when we view $p_\lact$ as a graph homomorphism, 

   $p_\lact\in \mathcal F$. Since $\mathcal F = \{ \id\}$, $p_\lact$ acts as the identity on edges (and vertices) of $G(\Gamma)$. Proposition \ref{prop:edges-enough} therefore implies that $\Gamma$ is stable. \qedhere

\end{proof}

\section{Cycle Graph Factors} 
\label{sec:cycle-factor}

From Proposition \ref{dis-union-prop} and  Theorem \ref{thm: rho isom}, we know that if $\Omega$ is a quasi-product with  quasi-factor $\Lambda$ and stable quasi-factor $\Gamma$, then $\Omega$ contains $\Gamma^0$-many isomorphic copies of $\Lambda$, with $\ract$ implementing the isomorphisms between the copies. In this section, we study how, when $C^*(\Gamma)$ can be viewed as a group $C^*$-algebra, then $\ract$ yields a semidirect product structure on $C^*(\Omega)$.
That is, in this setting, a quasi-product becomes a crossed product; see Theorem \ref{thm:circle-xprod} below.

We begin by recalling the definition of a $k$-graph $C^*$-algebra. In the following definition, 
and the rest of this section,
we follow Kumjian and Pask \cite{kp} and assume that our $k$-graphs are {\em row-finite} (that is, $v\Lambda^n$ is finite for all $v \in \Lambda^0$ and $n\in \N^k$) and {\em source-free} (that is, $v\Lambda^n \not= \emptyset$ for all $v, n$).

\begin{dfn}(cf.~\cite[Definitions 1.5]{kp}, \cite[Theorem C.1]{rsy-jfa}, \cite[Definition 7.4]{kps-hlogy})
\label{def:CK-relns}
    Let $\Lambda$ be a row-finite, source-free $k$-graph, and let $E = \{ e_1, \ldots, e_k\}$ denote the standard generators of $\N^k$.  A {\em Cuntz--Krieger $\Lambda$-family} in a $C^*$-algebra $B$ is a collection of partial isometries $\{ S_v: v \in \Lambda^0\} \cup \{ S_e: e \in \Lambda^E\} \subseteq B$ such that:
    \begin{enumerate}
        \item[(CK1)] $\{ S_v: v \in \Lambda^0\}$ is a set of mutually orthogonal projections.
        \item[(CK2)] Whenever $e, e', f, f' \in \Lambda^E$ satisfy $e f \sim f' e'$, then $S_e S_f = S_{f'} S_{e'}.$
        
        \item[(CK3)] For all $e \in \Lambda^E$, $S_e ^* S_e = S_{s(e)}$.
        \item[(CK4)] For all $1\leq i \leq k $ and all $v \in \Lambda^0$, 
        $\displaystyle S_v = \sum_{e \in v\Lambda^{e_i}} S_e S_e^*.$
    \end{enumerate}
   The universal $C^*$-algebra generated by a family of partial isometries $\{ S_v: v\in \Lambda^0\} \cup \{ S_e: e \in \Lambda^E\}$ is the $k$-graph $C^*$-algebra $C^*(\Lambda)$. 
\end{dfn}

   When $p = p_1 \cdots p_n \in G(\Lambda)^*$ represents $\lambda \in \Lambda$, (CK2) implies that $S_\lambda:= S_{p_1} \cdots S_{p_n}$ is well defined.  Furthermore, straightforward computations using the Cuntz--Krieger relations (CK1-4) show that $\{ S_\lambda S_\mu^*: \lambda, \mu \in \Lambda\}$ densely spans $C^*(\Lambda)$.

\begin{dfn}

 \label{cycle-defn}
Let $C_{n,k_2}$ be the (uniquely defined) $k_2$-graph whose 1-skeleton consists of $k_2$ parallel $n$-cycles, one of each color.  (There is only one possible choice of factorization rule on such a 1-skeleton.) Writing $C_{n, k_2}^0 = \{ w_0, \ldots, w_{n-1}\}$, the $j$th $n$-cycle has edges $\{f_{0,j},\dots,f_{n-1,j}\}$  such that $s(f_{i,j})=w_i$ and $r(f_{i,j})=w_{i+1}$ (mod $n$).
\end{dfn}

\begin{figure}[H]
	\centering
	\stepcounter{thm}
	\scalebox{.8}{
	\begin{tikzpicture}[roundnode/.style={circle, draw=black!60, fill=gray!5, very thick, minimum size=3mm}]
           \node[roundnode]        (w)                         {};
           \node[roundnode]        (x)       [below=2 of w]  {};
           \node[roundnode]        (v)       [right= 2 of w]  {};
           \node[roundnode]        (y)     [right=2 of x]      {};
           \draw[<-, line width=1.1, color=red, dashed] (w.south west) to[bend right=15] (x.north);
           \draw[<-, color=blue,line width=1.1, dotted] (w.south east) to[bend left=15] (x.north);
           \draw[<-] (w.south) to[] (x.north);

           \draw[->, line width=1.1, color=red, dashed] (v.south) to[bend left=15] (y.north east);
           \draw[->, color=blue,line width=1.1, dotted] (v.south) to[bend right=15] (y.north west);
           \draw[->] (v.south) to[] (y.north);

           \draw[->, line width=1.1, color=red, dashed] (w.east) to[bend left=15] (v.north west);
           \draw[->, color=blue,line width=1.1, dotted] (w.east) to[bend right=15] (v.south west);
           \draw[->] (w.east) to[] (v.west);

           \draw[<-, line width=1.1, color=red, dashed] (x.south east) to[bend right=15] (y.west);
           \draw[<-, color=blue,line width=1.1, dotted] (x.north east) to[bend left=15] (y.west);
           \draw[<-] (x.east) to[] (y.west);
           \end{tikzpicture}
       }
       \caption{Above is pictured $C_{4,3}$. Note that the factorization rule is 
       uniquely determined. This trait persists for arbitrary $C_{n,k_2}$.}
   \end{figure}
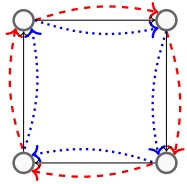

 In this section we will focus our attention on quasi-products $\Omega$ with quasi-factors $\Lambda$ (for an arbitrary $k_1$-graph $\Lambda$) and $C_{n,k_2}$. Note, for all $n,k_2\in \N$ the higher-rank graph $C_{n,k_2}$ satisfies the hypothesis of Lemma \ref{lem: stab-guarantee}. This means that $C_{n,k_2}$ is always a stable quasi-factor. In particular, thanks to Theorem \ref{thm: rho isom}, this ensures that each $(\widehat{f}_{i,j})_\ract$ is an isomorphism of $k_1$-graphs. Thus, there exists a $*$-isomorphism $\rho_{i,j}:C^*(\Omega_1^{w_{i+1}})\to C^*(\Omega_2^{w_i})$ induced by $(\widehat{f}_{i,j})_\ract$.

\begin{lem}
\label{lem: aut oplus}
    Let $\Omega$ be a quasi-product with quasi-factors $\Lambda$ and $C_{n,k_2}$.
    For each fixed  $1\leq j \leq k_2$, the family of isomorphisms $\left \{ \rho_{i,j}:C^*(\Omega_1^{w_{i+1}})\to C^*(\Omega_1^{w_{i}})\right \}_{i=1}^n$ induced by $(\widehat{f}_{i,j})_\ract$ yields $\displaystyle 
        \rho_{j} \in \Aut \left( \bigoplus_{i=0}^{n-1} C^*(\Omega_1^{w_i}) \right) 
       $
        given by 
        \begin{align*}  \rho_j(a_0,a_1,\cdots,a_{n-1}) = & (\rho_{0,j}(a_{1}),\rho_{1,j}(a_{2}),\cdots, \rho_{n-1,j}(a_{0})).
    \end{align*}  
    
\end{lem}

\begin{proof}
    From Theorem \ref{thm: rho isom}, $(\widehat{f}_{i,j})_\ract: \Omega_1^{w_{i+1}}\to\Omega_1^{w_i}$ is an isomorphism of $k_1$-graphs. Therefore, by \cite[Corollary 3.5]{kp},  
    it induces a $*$-isomorphism $\rho_{i,j}: C^*(\Omega_1^{w_{i+1}})\to C^*(\Omega_1^{w_i})$.

    That $\{\rho_{i,j}\}_{0\leq i \leq n-1}$ defines the stated automorphism $\widehat \rho_j$ of $\bigoplus_i C^*(\Omega_1^{w_i})$ hinges primarily on the fact that $\|\rho_j(a)\| = \|\rho_{i,j}(a_i)\|$ for some $0\leq i \leq n-1$. Thus, continuity is inherited directly from the continuity of the isomorphisms $\rho_{i,j}$.
\end{proof}

Recall from Proposition \ref{dis-union-prop} that $\Omega_1 = \bigsqcup_{i=0}^{n-1} \Omega_1^{w_i} \times \{ w_i\}$. Thus,  $C^*(\Omega_1)\cong \bigoplus_{i} C^*(\Omega_1^{w_i})$. Moreover, since $\Omega_1 = d^{-1}(\N^{k_1})$, it is straightforward to check that $C^*(\Omega_1)$ is a sub-$C^*$-algebra of $C^*(\Omega)$.

Our primary result of this section (Theorem \ref{thm:circle-xprod}) shows that if $\Omega $ is a quasi-product of $\Lambda$ and $C_{n,k_2}$  and $|\Lambda^0|<\infty$, the automorphism $\rho_j$ defined above is induced on elements of $C^*(\Omega_1)$ via conjugation by
\begin{equation} U_j = \sum_{v\in \Lambda^0} \sum_{i=0}^{n-1} S_{(v,f_{i,j})}^*.
\label{eq:Uj}
\end{equation}

\begin{thm}
    Let $\Omega$ be a quasi-product with quasi-factors $\Lambda$ and $C_{n,k_2}$. If $\Lambda^0$ is finite then
    \[ C^*(\Omega) \cong \left( \bigoplus_{i=0}^{n-1} C^*(\Omega_1^{w_i}) \right)\rtimes_{\rho} \Z^{k_2}  \cong \left(\bigoplus_{i=1}^n C^*(\Lambda) \right) \rtimes_\rho \Z^{k_2}\]
    with $\rho_x = \rho_{1}^{x_1}\cdots\rho_{k_2}^{x_{k_2}}$.
    \label{thm:circle-xprod}
\end{thm}

\begin{proof}
    Firstly, we will demonstrate that
  $\{ U_j\}_{1\leq j \leq k_2}$
    are a commuting family of unitaries, that is, a unitary representation of $\Z^{k_2}$. For this claim, we rely on the Cuntz-Krieger relations. 
    To see that $U_j U_\ell = U_\ell U_j,$  recall that $S^*_{(v,f_{m,j})}S^*_{(u,f_{i,\ell})}= \delta_{v,u} \delta_{m,i-1} S^*_{(v,f_{i,j}f_{i-1,\ell})}$. Since $f_{i,j}f_{i-1,\ell}=f_{i,\ell}f_{i-1,j}\in C_{n,k_2}$,
    \begin{align*}
        U_jU_\ell &= \left(\sum_{v\in \Lambda^0} \sum_{m=0}^{n-1} S_{(v,f_{m,j})}^*\right)\left(\sum_{u\in \Lambda^0} \sum_{i=0}^{n-1} S^*_{(u,f_{i,\ell})}\right) \\
        & = \sum_{v\in \Lambda} \sum_{i=0}^{n-1} S^*_{(v,f_{i,j}f_{i-1,\ell})}  = \sum_{v\in \Lambda} \sum_{i=0}^{n-1} S^*_{(v,f_{i,\ell}f_{i-1,j})}\\
        & = U_\ell U_j.
    \end{align*}
 
    To see that $U_j$ is a unitary, we observe that $S_{(v,f_{i-1,j})}S_{(v,f_{m-1,j})}^* 
    = \delta_{i,m}S_{(v,w_i)} = S_{(v,f_{i,j})}^* S_{(v,f_{m,j})}$. 
    Furthermore, $|\Lambda^0| < \infty \implies |\Omega^0| < \infty$; the universality of $C^*(\Omega)$ therefore implies that $1 = \sum_{v\in \Lambda^0} \sum_{i=0}^{n-1} S_{v,w_i}$. Consequently, 
    \begin{align*}
        U_jU_j^* &= \left(\sum_{v\in \Lambda^0} \sum_{i=0}^{n-1} S^*_{(v,f_{i,j})}\right)\left(\sum_{u\in \Lambda^0} \sum_{m=0}^{n-1} S_{(u,f_{m,j})}\right) \\
        &= \sum_{v\in \Lambda} \sum_{i=0}^{n-1} S^*_{(v,f_{i,j})}S_{(v,f_{i,j})} = \sum_{v\in \Lambda} \sum_{i=0}^{n-1} S_{v,w_i} \\
        &= 1 = \left(\sum_{v\in \Lambda^0} \sum_{i=0}^{n-1} S_{(v,f_{i,j})}\right)\left(\sum_{u\in \Lambda^0} \sum_{m=0}^{n-1} S_{(u,f_{m,j})}^*\right)\\
        &=U_j^*U_j.
    \end{align*}
    We conclude that $\{U_j\}$ is a unitary representation of $\Z^{k_2}$ contained in $C^*(\Omega)$.

  To see that conjugation by $U_j$ will induce the action of $\rho_{j}$, fix $\lambda \in \Omega_1^{w_m} (\cong \Lambda ) .$ By Lemma \ref{k-graph-mp-lem},
   \[(\lambda, w_m) (s(\lambda), f_{m-1, j})  = (r(\lambda),f_{m-1,j})((\widehat{f}_{m-1,j})_\ract(\lambda),w_{m-1})\]
  since $C_{n, k_2}$ is a stable quasi-factor.  (CK2) now implies that 
    \begin{align*}
        U_j S_{(\lambda, w_m)}U_j^*&=\left(\sum_{i=0}^{n-1}\sum_{v\in \Lambda^0} S_{(v,f_{i,j})}^*\right) S_{(\lambda, w_m)} \left(\sum_{\ell =0}^{n-1}\sum_{u\in \Lambda^0} S_{(u,f_{\ell,j})}\right)\\
        &= S_{(r(\lambda),f_{m-1,j})}^*S_{(\lambda,w_m)}\left(\sum_{\ell=0}^{n-1}\sum_{u\in \Lambda^0} S_{(u,f_{\ell,j})}\right)\\
        & = S_{(r(\lambda),f_{m-1,j})}^*S_{(\lambda,w_m)}S_{(s(\lambda),f_{m-1,j})}\\
        &= S_{(r(\lambda), f_{m-1,j})}^* S_{(r(\lambda), f_{m-1,j})}S_{((\widehat{f}_{m-1,j})_\ract(\lambda), w_m)}\\
        & 
        =S_{((\widehat{f}_{m-1,j})_\ract(\lambda), w_m)} =  \rho_{m-1, j}(S_{(\lambda, w_m)}).
	\end{align*}

As each $ \rho_{i,j}$ is a $*$-homomorphism, and $C^*(\Omega_1^{w_m})$ is generated by $\{ S_{(\lambda, w_m)}\}_\lambda$, we conclude that the automorphism $\rho_j$ of $\bigoplus_{i=0}^{n-1} C^*(\Omega_1^{w_i})$ is indeed given by conjugation by $U_j$.

    This demonstrates that $\{U_j\}_{j=1}^{k_2}$ yields a unitary representation of $\Z^{k_2}$ which induces $\rho$ on $C^*(\Omega_1) \cong \bigoplus_{i=0}^{n-1} C^*(\Omega_1^{w_i})$ via conjugation. The universal property of the crossed product now gives a $*$-homomorphism 
    \[ \Phi: \left( \bigoplus_{i=0}^{n-1} C^*(\Omega_1^{w_i}) \right)\rtimes_{\rho} \Z^{k_2} \to C^*(\Omega)\]
    with image  $\langle C^*(\Omega_1),U_j\rangle$.  
   It remains to show that $\Phi$ is bijective.

    We demonstrate surjectivity 
    by observing first that each vertex projection $S_{(v, w_i)}$ lies in  
    $C^*(\Omega_1)$ and therefore $S_{(v,w_{i+1})}U_j^*=S_{(v,f_{i,j})} \in \text{Im}\, \Phi$ for all $v,i,j$. In particular, every generator of $C^*(\Omega_2)$ lies in $\text{Im}\, \Phi$,
    making the image all of $C^*(\Omega)$.

    To demonstrate injectivity, we 
    construct a Cuntz--Krieger $\Omega$-family inside of $ \left( \bigoplus_i C^*(\Omega_1^{w_i})\right) \rtimes_{ \rho} \Z^{k_2}$, such that the associated $*$-homomorphism 
    $\Psi: C^*(\Omega) \to \left( \bigoplus_i C^*(\Omega_1^{w_i})\right) \rtimes_{ \rho} \Z^{k_2}$ satisfies $\Psi\Phi=\operatorname{id}$.
      To that end, write $V_j$ for the canonical generators of $\Z^{k_2}$ in the crossed product, and for each  $e \in G(\Lambda)^1, f_{i,j} \in C_{n, k_2}^1$ and each vertex $v \in \Lambda^0, 
      w_i \in C_{n,k_2}^0,$  define $T_{(e, w_i)} , T_{(v, w_i)}, T_{(v, f_{i,j})}\in  \left( \bigoplus_i C^*(\Omega_1^{w_i})\right) \rtimes_{ \rho} \Z^{k_2} $ by 
      \[ T_{(e, w_i)} = S_{(e, w_i)},   
      \quad T_{(v, w_i)} = S_{(v, w_i)} , \quad  T_{(v,f_{i,j})} = V_{j}^*T_{(v,w_i)}.\]

We now show that  $\{ T_\lambda: \lambda \in \Omega^E \cup \Omega^0\}$ is a Cuntz--Krieger $\Omega$-family within  $\left( \bigoplus_i C^*(\Omega_1^{w_i})\right) \rtimes_{ \rho} \Z^{k_2} $. To that end, 
observe first that 
   (CK1) is ensured since
   the vertex projections $S_{(v,w_i)}$ in $\bigoplus_i C^*(\Omega_1^{w_i})$ are orthogonal.

Before checking (CK2), we observe that for any vertex projection, 
\[ V_j S_{(v,w_i)} V_j^* =  \rho_j(S_{(v,w_i)}) =  \rho_{i-1, j}(S_{(v, w_i)}) = S_{(v, w_{i-1})}.\]
Moreover, the unitaries $V_j$ commute, since they represent the Abelian group $\Z^{k_2}$.  Consequently, 
\begin{align*}
    T_{(v, f_{i,j} )}T_{(v, f_{i-1, \ell})} &= V_j^* S_{(v, w_i)} V_\ell^* S_{(v, w_{i-1})} = S_{(v, w_{i+1})} V_j^* V_\ell^* S_{(v, w_{i-1})}  \\
    &= S_{(v, w_{i+1})} V_\ell^* V_j^* S_{(v, w_{i-1})}= V_\ell ^* S_{(v, w_i)} V_j^* S_{(v, w_{i-1})} \\
    & = T_{(v, f_{i, \ell})} T_{(v, f_{i-1, j})}.
\end{align*}
That is, (CK2) holds for the generators $T_{(v, f_{i,j})}$ associated to $C_{n, k_2}.$
 We know that (CK2) holds among the generators $T_{(e, w_i)}$ since (CK2) holds for $\bigoplus_i C^*(\Omega_1^{w_i}) \cong \bigoplus_i C^*(\Lambda)$ by construction.  To see that $T_{(e, w_{i+1})} T_{(s(e), f_{i, j})} = T_{(r(e), f_{i, j})} T_{((\widehat{f}_{i,j})_\ract(e), w_i)}$,  observe that 
 \begin{align*}
T_{(r(e), f_{i,j})} T_{((\widehat{f}_{i,j})_\ract(e), w_i)}&= V_j^* S_{(r(e), w_i)} S_{((\widehat{f}_{i,j})_\ract(e), w_i)} = V_j^* S_{((\widehat{f}_{i,j})_\ract(e), w_i)}  \\
\text{ while  } \quad T_{(e, w_{i+1})} T_{(s(e), f_{i,j})} & = S_{(e, w_{i+1})} V_j^* S_{(s(e), w_i)} \\
& = S_{(e, w_{i+1})} S_{(s(e), w_{i+1})} V_j^* = S_{(e, w_{i+1})} V_j^*.
 \end{align*}
As the unitaries $V_j$ implement the action $ \rho_{i,j}$ which is induced by $(\widehat{f}_{i,j})_\ract,$ we have $V_j S_{(e, w_{i+1})} V_j^* = S_{((\widehat{f}_{i,j})_\ract(e), w_i)}$; that is, (CK2) holds.

  Observe that (CK3) and (CK4) are satisfied for $\{ T_{(e, w_i)}\}$ since they hold in $C^*(\Omega^{w_i}_1)$.
 Checking (CK3) for $T_{(v, f_{i,j})}$ amounts to observing that each $V_j$ is a unitary:
 \[ T_{(v, f_{i,j})}^* T_{(v, f_{i,j})} = T_{(v, w_i)}^* V_j  V_{j}^*T_{(v,w_i)} = T_{(v, w_i)} = T_{s(v, f_{i,j})}. \]
  For (CK4), recall that for each $i,j$ there is only one edge  in $\Omega_2$ with range $(v, w_{i+1})$ and degree $e_{j+k_1},$ so 
  \[ \sum_{e \in (v, w_{i+1})\Omega^{e_{j+k_1}}} T_e T_e^* = T_{(v, f_{i,j})} T_{(v, f_{i,j})}^* 
  = V_j^* T_{(v, w_i)} V_j^* = T_{(v, w_{i+1})}.\]
    We conclude that $\{ T_e: e \in G(\Omega)^1\}$ is a Cuntz--Krieger $\Omega$-family. Thus,  the universal property of $C^*(\Omega)$ yields a $*$-homomorphism $\Psi:C^*(\Omega)\to   \left( \bigoplus C^*(\Omega_1^{w_i}) \right)\rtimes_{\rho} \Z^{k_2}$ satisfying $\Psi(S_e) = T_e$. In particular, 
     as the sum of all vertex projections is the identity operator, 
     it can be easily checked that $\Psi \circ \Phi = \operatorname{id}$ on the generators of the crossed product. Hence, $\Phi$ is an isomorphism as claimed.
\end{proof}

We close this section by drawing attention back to Example \ref{ex: bouq}. If we call the 1-graph given by the black edges $\Lambda$ and the 2-graph given by the blue  and red edges $\Gamma = C_{1,2}$, then the 2-graphs in Example \ref{ex: bouq} are all  quasi-products $\Omega_i$ with quasi-factors $\Lambda$ and $C_{1,2}$. From Theorem \ref{thm:circle-xprod}, we obtain the isomorphism $C^*(\Omega) \cong C^*(\Lambda)\rtimes_{\rho} \Z^2 = \mathcal O_3 \rtimes_\rho\Z^2$. 
If the factorization rule on $\Omega$ is given by $\sim_1$, 
$\rho=\id$, whereas in the case of  $\sim_2$ we have $e_{\ract}(f_1)=f_2$ and $g_{\ract}(f_1)=f_2$, i.e., $\rho_{(1,0)}(s_1) = s_2=\rho_{(0,1)}(s_1)$.  

\section{Stability and products}
\label{sec:qp=is-product}
In this section, we use the framework of stable quasi-products to characterize when a $k$-graph is isomorphic to a product graph. As we discussed in the introduction, in a product graph $\Omega = \Gamma \times \Lambda,$ both factors are stable.  However, Example \ref{ex: path loops} below shows that a $k$-graph may be isomorphic to a product graph without   being stable.  With Example \ref{ex: path loops} in mind, we define {\em relaxed stable}  quasi-products (see Theorem \ref{thm:relaxed stable}).  Theorem \ref{thm:product-isom} uses this concept to characterize when  a quasi-product $k$-graph is isomorphic to a product graph.

\subsection{Motivating examples}
\label{sec:examples-5}
\begin{example}
\label{ex: path loops}

Let  $\Lambda$ be the 1-graph 
\[ \Lambda = \stackrel{\cdot}{\scriptstyle w_0} \stackrel{\longleftarrow}{\scriptstyle  e_0} \stackrel{\cdot}{\scriptstyle  w_1}\stackrel{\longleftarrow}{\scriptstyle e_1} \stackrel{\cdot}{\scriptstyle  w_2} \cdots \]
and $\Gamma $  the 1-graph with one vertex $v$ and two edges $f, g$.  The one-skeleton of the product is as follows:

\begin{center}
    \scalebox{.8}{
	\begin{tikzpicture}[xscale=2]
            
            \foreach \x in {0,...,5}{
                \node[] (v) at (\x,0){$\scriptstyle{(v,w_\x)}$};
                \draw[<-] (v) to[] node[inner sep=3pt,black,above]{$\scriptstyle{( e_\x, v)}$} (\x+.75,0);
                \draw[->,dashed] (v) to[in=75,out=105,looseness=25] node[inner sep=3pt,black,above]{$\scriptstyle{(w_\x, f)}$} (v);
                \draw[->,dashed] (v) to[in=-75,out=-105,looseness=25] node[inner sep=3pt,black,below]{$\scriptstyle{(w_\x, g)}$} (v);
            }
            \node[] at (6,0) {$\cdots$};
            \end{tikzpicture}
        }
    
\end{center}
Thanks to Lemma \ref{lem: stab-guarantee}, $\Lambda$ is always a (left) stable quasi-factor. 
Define the quasi-product $\Omega$ by setting
\[ (e_{2i}, v)( w_{2i+1}, f)=( w_{2i}, f)( e_{2i}, v) , \qquad (e_{2i+1}, v) (w_{2i+2}, f)=(w_{2i+1}, g)( e_{2i+1}, v) . \]
Then $\Gamma$ is not a stable quasi-factor, and in particular, $\Omega \not= \Lambda \times \Gamma$.  However, there is an isomorphism $\psi: \Omega \to \Lambda \times \Gamma$: we have $\psi(e_n, v) = (e_n, v)$ for all $n$, 

$\psi(w_j, f)=  (w_j, g)$ when $j= 0,1 \pmod 4$ and $\psi (w_j, f)= (w_j, f)$ otherwise.

To see that $\psi$ is indeed an isomorphism of 2-graphs, we must check that it preserves all commuting squares. First, if $i = 2m$, 
\begin{align*}
  \psi( w_{2i}, f) \psi( e_{2i}, v) &= ( w_{2i}, g) (e_{2i}, v) \sim_\times ( e_{2i}, v) ( w_{2i+1}, g) = \psi( e_{2i}, v) \psi(w_{2i+1}, f),
\end{align*}
that is, $\psi$ preserves the commuting square $( w_{2i}, f)( e_{2i}, v) \sim_\Omega ( e_{2i}, v)(w_{2i+1}, f).$  Similarly, 
\begin{align*} \psi( w_{2i+1},f) \psi( e_{2i+1}, v)& = (w_{2i+1}, g) (e_{2i+1},v) \sim_\times ( e_{2i+1}, v) (w_{2i+2}, g) \\
& = \psi( e_{2i+1}, v) \psi( w_{2i+2}, g),
\end{align*}
so $\psi$ respects the commuting square $(w_{2i+1}, f)( e_{2i+1}, v) = ( e_{2i+1}, v) (w_{2i+2}, g).$

On the other hand, if $i = 2m +1$ is odd, then 
\[ \psi( w_{2i}, f) \psi( e_{2i}, v) = ( w_{2i}, f) ( e_{2i}, v) \sim_\times ( e_{2i}, v) (w_{2i+1}, f) = \psi(e_{2i}, v) \psi( w_{2i+1}, f),\]
so $\psi$ again preserves the commuting square $( w_{2i}, f)(e_{2i}, v) \sim_\Omega (e_{2i}, v)(w_{2i+1}, f);$   and since $2i+2 = 0 \pmod 4,$ 
\begin{align*} 
\psi(w_{2i+1}, f) \psi(e_{2i+1}, v)& = ( w_{2i+1}, f) ( e_{2i+1}, v) \sim_\times ( e_{2i+1}, v) (w_{2i+2}, f)\\
& = \psi( e_{2i+1}, v) \psi( w_{2i+2}, g),
\end{align*}
so $\psi$ respects the commuting square $( w_{2i+1}, f)(e_{2i+1}, v) = ( e_{2i+1}, v) (w_{2i+2}, g).$
\end{example}

 Example \ref{ex: path loops}  might seem to suggest that whenever the quasi-factor $\Lambda$ of a quasi-product $\Omega$ is guaranteed to be stable, we have an isomorphism $\Omega \cong \Lambda \times \Gamma.$ The following example shows that this is not the case.

\begin{example}
\label{ex: undirected counter}

Consider the one-skeleton of a product graph below.

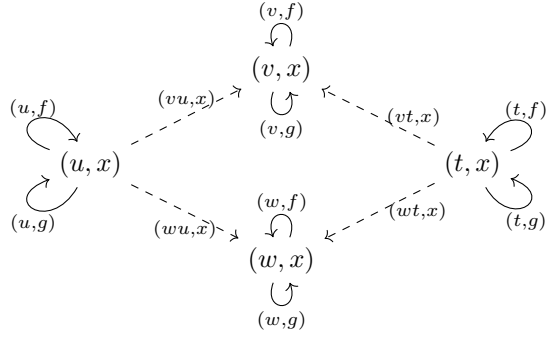
\begin{center}
    
\begin{tikzpicture}[x=1in,y=1in]
    \node[] (u) at (0,0) {$(u,x)$};
    \node[] (v) at (1,.5) {$(v,x)$};
    \node[] (w) at (1,-.5) {$(w,x)$};
    \node[] (t) at (2,0) {$(t,x)$};

    \draw[<-] (u) to[out=120, in=160, looseness=5] node[above, inner sep=1pt] {$\scriptstyle{(u,f)}$} (u);
    \draw[->] (u) to[out=240, in=200, looseness=5] node[below, inner sep=2pt] {$\scriptstyle{(u,g)}$} (u);

    \draw[<-] (t) to[out=180-120, in=180-160, looseness=5] node[above, inner sep=1pt] {$\scriptstyle{(t,f)}$} (t);
    \draw[->] (t) to[out=180-240, in=180-200, looseness=5] node[below, inner sep=2pt] {$\scriptstyle{(t,g)}$} (t);

    \draw[<-] (v) to[out=110, in=70, looseness=5] node[above, inner sep=1pt] {$\scriptstyle{(v,f)}$} (v);
    \draw[->] (v) to[out=-110, in=-70, looseness=5] node[below, inner sep=2pt] {$\scriptstyle{(v,g)}$} (v);

    \draw[<-] (w) to[out=110, in=70, looseness=5] node[above, inner sep=1pt] {$\scriptstyle{(w,f)}$} (w);
    \draw[->] (w) to[out=-110, in=-70, looseness=5] node[below, inner sep=2pt] {$\scriptstyle{(w,g)}$} (w);

    \draw[->, dashed] (u) to[out=26.57, in=206.57] node[above, inner sep=3pt] {$\scriptstyle{(vu,x)}$} (v);
    \draw[->, dashed] (u) to[out=-26.57, in=153.43] node[below, inner sep=3pt] {$\scriptstyle{(wu,x)}$} (w);
    \draw[->, dashed] (t) to[out=180-26.57, in=180-206.57] node[right, inner sep=3pt] {$\scriptstyle{(vt,x)}$} (v);
    \draw[->, dashed] (t) to[out=180+26.57, in=180-153.43] node[right, inner sep=3pt] {$\scriptstyle{(wt,x)}$} (w);

\end{tikzpicture}
\end{center}

Here, $\Lambda$ is the 1-graph corresponding to the dashed edges, and $\Gamma$ is once again the 1-graph with one vertex $x$ and two loops $f, g$.

Once again, the left factor $\Lambda$ is stable by Lemma \ref{lem: stab-guarantee}.  There are (at least) two ways to construct a non-stable quasi-product with this 1-skeleton; one of them is isomorphic to $\Lambda \times \Gamma$ and the other is not.

We define the first quasi-product $\Omega_1$ by defining the following commuting squares.  Notice that this data uniquely determines the remaining commuting squares in $\Omega_1$; as we only have two colors of edges, Remark \ref{rmk:commuting-squares} tells us that $\Omega_1$ is completely determined by specifying its commuting squares.
\begin{center}
    \begin{tikzpicture}[x=2cm, y=2cm]
    \node[] (tu) at (0,0) {$(u,x)$};
    \node[] (tv) at (1,0) {$(v,x)$};
    \node[] (tt) at (2,0) {$(t,x)$};
    \node[] (tw) at (3,0) {$(w,x)$};
    \node[] (tu2) at (4,0) {$(u,x)$};

    \node[] (bu) at (0,-1) {$(u,x)$};
    \node[] (bv) at (1,-1) {$(v,x)$};
    \node[] (bt) at (2,-1) {$(t,x)$};
    \node[] (bw) at (3,-1) {$(w,x)$};
    \node[] (bu2) at (4,-1) {$(u,x)$};

    \draw[->, dashed] (tu) to[in=180,out=0] (tv);
    \draw[<-, dashed] (tv) to[in=180,out=0] (tt);
    \draw[->, dashed] (tt) to[in=180,out=0] (tw);
    \draw[<-, dashed] (tw) to[in=180,out=0] (tu2);

    \draw[->, dashed] (bu) to[in=180,out=0] (bv);
    \draw[<-, dashed] (bv) to[in=180,out=0] (bt);
    \draw[->, dashed] (bt) to[in=180,out=0] (bw);
    \draw[<-, dashed] (bw) to[in=180,out=0] (bu2);

    \draw[->] (tu) to[out=270,in=90] node[right] {$\scriptstyle{(u,f)}$} (bu);
    \draw[->] (tv) to[out=270,in=90] node[right] {$\scriptstyle{(v,f)}$} (bv);
    \draw[->] (tt) to[out=270,in=90] node[right] {$\scriptstyle{(t,f)}$} (bt); 
    \draw[->] (tw) to[out=270,in=90] node[right] {$\scriptstyle{(w,g)}$} (bw);
    \draw[->] (tu2) to[out=270,in=90] node[right] {$\scriptstyle{(u,g)}$} (bu2);
\end{tikzpicture}
\end{center}

If we had an isomorphism $\varphi: \Omega_1 \to \Lambda \times \Gamma$ with $\varphi(u,g) = (u,g)$, in order for $\varphi$ to preserve the commutation relation $\sim_1$ indicated in the diagram above, we would have to have  $\varphi(u,f) = (u,g),$ contradicting the injectivity of $\varphi$.  Similarly, if $\varphi(u, g) = (u, f),$ the fact that $\varphi: \Omega_1 \to \Lambda \times \Gamma$ must respect the commuting squares forces $\varphi(u, f) = (u, f)$.  That is, no  isomorphism $\varphi: \Omega_1 \to \Lambda \times \Gamma$ exists.

On the other hand, we obtain a different 2-graph $\Omega_2$ by specifying the following commuting squares: 

\begin{center}
    \begin{tikzpicture}[x=2cm, y=2cm]
    \node[] (tu) at (0,0) {$(u,x)$};
    \node[] (tv) at (1,0) {$(v,x)$};
    \node[] (tt) at (2,0) {$(t,x)$};
    \node[] (tw) at (3,0) {$(w,x)$};
    \node[] (tu2) at (4,0) {$(u,x)$};

    \node[] (bu) at (0,-1) {$(u,x)$};
    \node[] (bv) at (1,-1) {$(v,x)$};
    \node[] (bt) at (2,-1) {$(t,x)$};
    \node[] (bw) at (3,-1) {$(w,x)$};
    \node[] (bu2) at (4,-1) {$(u,x)$};

    \draw[->, dashed] (tu) to[in=180,out=0] (tv);
    \draw[<-, dashed] (tv) to[in=180,out=0] (tt);
    \draw[->, dashed] (tt) to[in=180,out=0] (tw);
    \draw[<-, dashed] (tw) to[in=180,out=0] (tu2);

    \draw[->, dashed] (bu) to[in=180,out=0] (bv);
    \draw[<-, dashed] (bv) to[in=180,out=0] (bt);
    \draw[->, dashed] (bt) to[in=180,out=0] (bw);
    \draw[<-, dashed] (bw) to[in=180,out=0] (bu2);

    \draw[->] (tu) to[out=270,in=90] node[right] {$\scriptstyle{(u,f)}$} (bu);
    \draw[->] (tv) to[out=270,in=90] node[right] {$\scriptstyle{(v,f)}$} (bv);
    \draw[->] (tt) to[out=270,in=90] node[right] {$\scriptstyle{(t,f)}$} (bt); 
    \draw[->] (tw) to[out=270,in=90] node[right] {$\scriptstyle{(w,g)}$} (bw);
    \draw[->] (tu2) to[out=270,in=90] node[right] {$\scriptstyle{(u,f)}$} (bu2);
\end{tikzpicture}

\end{center}

Although the rightmost square above shows that $\Omega_2 \not = \Lambda \times \Gamma$, setting 
\[ \varphi(u, f) = (u, g); \quad \varphi(v, f) = (v, g); \quad \varphi(t,f) = (t, g); \quad \varphi(w, f) = (w, f)\]
yields an isomorphism of 2-graphs $\varphi: \Omega_2 \to \Lambda \times \Gamma$.
(The proof is analogous to the proof that $\psi$ from Example \ref{ex: path loops} is an isomorphism.)
\end{example}

\subsection{Relaxed stability and products}
\label{sec:relaxed-stability}

The key difference between Example \ref{ex: path loops} and Example \ref{ex: undirected counter} lies in the structure of the undirected graph underlying $\Lambda$. In both cases, there is at most one path in $\Lambda$ between any two vertices, and so one might hope that the isomorphisms implemented by $\lact$ between the various components of $\Omega_2 = \bigsqcup_{v \in \Lambda^0} \Omega_2^x$ could be untwisted to yield an isomorphism $\Omega \cong \Lambda \times \Gamma$. Essentially,  this is not possible in $\Omega_1$ from Example \ref{ex: undirected counter} because there are two {\em undirected} paths $p, \tilde p$ between $t, w$ which give rise to different isomorphisms $\Omega_2^t \cong \Omega_2^w$.

Thus, in order to describe when a quasi-product is ``stable in disguise'' or {\em relaxed stable}, we must first pause to study the underlying undirected graph of a $k$-graph (Definition \ref{def: bi-direct}).  Then, Theorem  \ref{thm:relaxed stable} characterizes which quasi-products are isomorphic to stable quasi-products.  Finally, we use this perspective to identify, in Theorem \ref{thm:product-isom}, precisely when a quasi-product is isomorphic to a product $k$-graph (and hence when its $C^*$-algebra has a tensor product decomposition).

\begin{dfn}
\label{def: bi-direct}
   Given a directed graph $G$, define the graph $\widehat G$ by

\begin{align*}
    \widehat G^0 &= G^0 & \widehat G^1&= \{ e_+, e_-: e \in G^1\} \\
    \hat s(e_+) = &s(e)=\hat r(e_-) & \hat r(e_+) &= r(e) = \hat s(e_-) 
\end{align*}
\end{dfn} 
 
If we suppose that $e_{\lact}$ is bijective for all $e\in G(\Lambda)$, 
we can extend $\lact$ to give us an isomorphism (which we will denote $\xi_\lact$) for any path $\xi \in \widehat G(\Lambda).$  Namely, if $\xi = e^1_\pm e^2_\pm \cdots e^n_\pm,$
\[ \xi_\lact := (e^1_\lact)^\pm \circ (e^2_\lact)^\pm \circ \cdots \circ (e^n_\lact)^\pm.\]

If $G$ is connected, then 
for any two vertices $v, w \in G^0,$ there exists a path $\xi \in v \widehat G^* w.$
In particular, $\Lambda$ connected implies that  for any $v \in \Lambda^0$ there exists a path $\xi^v \in \widehat G(\Lambda)^*$ with $r(\xi^v) =v$ and $s(\xi^v) = x_0$, where $\Gamma \cong \Omega_2^{x_0}$. 
Thus, if $\Omega$ is a quasiproduct for which $\Lambda$ is connected and $\xi^v_\lact$ is bijective for each $\xi \in \widehat G(\Lambda), v \in \Gamma^0,$ then $\lact$ implements an 
isomorphism $\Gamma \cong \Omega_2^v .$   
If moreover this isomorphism is independent of the choice of path in $\widehat G(\Lambda)$ linking $v$ and $x_0$, then the following Theorem establishes that we can patch together these isomorphisms to yield a global isomorphism of $\Omega$ with a stable quasiproduct $ \widetilde \Omega .$

\begin{thm}
\label{thm:relaxed stable}
Let $\Omega$ be a quasi-product with quasi-factors $\Lambda$ and $\Gamma$.  Assume that $\Lambda$ is connected. 
Suppose that for each $p \in G(\Lambda),$ $\widehat p_\lact $ is bijective, and 
that for any $\xi,\xi'\in \widehat{G}(\Lambda)$, $\widehat \xi_\lact=\widehat \xi'_\lact$ whenever $r(\xi)=r(\xi')$ and $s(\xi)=s(\xi')$. Then there is a $k$-graph isomorphism $\Theta: \Omega \cong \widetilde \Omega$, where $\widetilde \Omega$ is a stable quasi-product with quasi-factor $\Lambda$ and stable quasi-factor $\Gamma$.  
\end{thm}

In the setting of the Theorem, we say that $\Gamma$ is a {\em relaxed stable} quasi-factor of $\Omega$. 

We begin with a lemma.
\begin{lem}
    Under the hypotheses of Theorem \ref{thm:relaxed stable}, for any $v, w \in \Lambda^0$ we obtain a canonical isomorphism of $k_2$-graphs $\Omega_2^v \to \Omega_2^w$.  This isomorphism is implemented by $\widehat \xi_\lact$ for any $\xi \in v \widehat G(\Lambda)^* w$.
    \label{lem:xi}
\end{lem}
\begin{proof}
As $\widehat p_\lact$ is always well defined and degree-, source-, and range-preserving;  is bijective and depends only on $ r(p), s(p)$ by hypothesis; and $\Lambda$ is connected,
  the only thing  to check   is that $\widehat p_\lact $ respects composition in $\Omega_2^{s(p)}$.  To that end, suppose that $q = gh \in G(\Gamma)^*$ and $p=e^1_+ e^2_+\cdots e^n_+$. Recall from Proposition \ref{prop:source-range-general} that $\ell := g_\ract(p) \in G(\Lambda)^*$ has the same source and range as $p$, so $\widehat p_\lact = \widehat \ell_\lact$ by hypothesis.  Using (MP2), we see that 
  \begin{align*}  (r(p), [p_\lact(gh)]) &=  ([p], r(g)) \lact (s(p), [gh]) = ([p], r(g)) \lact \left( (s(p), [g])(s(p),[h])\right) \\
  &= (r(p), [p_\lact(g)]) ([\ell], r(h))\lact (s(\ell), [h]) \\
  & = (r(p), [p_\lact(g)]) (r(p), [\ell_\lact(h)]) = (r(p), [p_\lact(g) p_\lact(h)]);
  \end{align*}
  that is, $\widehat p_\lact$ implements an isomorphism $\Omega_2^{s(p)} \to \Omega_2^{r(p)}.$ For a general $\xi\in v\widehat{G}(\Lambda)^*w$, write $\xi=p^1_\pm\cdots p^n_\pm$ for $p^1, \ldots, p^n \in G(\Lambda)^*.$ Then $\widehat \xi_\lact = (\widehat p^1_\lact)^\pm \circ \cdots \circ (\widehat p^n_\lact)^\pm$ is  given by a  composition of $k_2$-graph isomorphisms, which is canonical in the sense that it only depends on $s(\xi), r(\xi)$. 
\end{proof}

To construct the $k$-graph $\widetilde \Omega$ of  Theorem \ref{thm:relaxed stable}, we will follow \cite[Lemma 3.28]{matched-pair}.  That is, we set $\widetilde \Omega_1 = \bigsqcup_{x \in \Gamma^0} \Omega_1^x \times \{x\}, \ \widetilde \Omega_2 = \bigsqcup_{v \in \Lambda^0} \{ v\} \times \Gamma$ as $k_i$-graphs.  With this information, \cite[Lemma 3.28]{matched-pair} tells us that to obtain a matched pair $(\widetilde \Omega_1, \widetilde \Omega_2, \widetilde \lact, \widetilde \ract)$ (and therefore a quasi-product $\widetilde \Omega$), it suffices to define $\widetilde \lact, \widetilde \ract$ and check that for any edges $e \in G(\Lambda)^1, g \in G(\Gamma)^1,$

the following conditions hold:

\begin{enumerate}
    \item[(K1)] 
    $s(e_{\tilde \lact}(g)) = s(g)$ and $r(g_{\tilde \ract}(e)) = r(e)$.
    \item[(K2)] 
    The right action $\widetilde \ract$ satisfies (MP3) for edges.  That is, for any $(e_1, e_2)\in G(\Lambda)^1 * G(\Lambda)^1$ and any $g \in G(\Gamma)^1$,
    \begin{align*}
       ([e_1 e_2], r(g))  & \widetilde \ract (s(e_2), g)\\
       & \stackrel{!}{=} (e_1, r(g)) \widetilde \ract \left( (e_2, r(g)) \widetilde \lact (s(e_2), g)\right) \left( ( e_2, r(g)) \widetilde \ract (s(e_2), g)\right) \\
       &= \left( ((e_2)_{\tilde \lact}(g))_{\tilde \ract}(e_1), s(g)) (g_{\tilde \ract}(e_2), s(g)\right)\\
       & = \left( \left[ ((e_2)_{\tilde \lact}(g))_{\tilde \ract}(e_1) g_{\tilde \ract}(e_2)\right], s(g) \right).
    \end{align*}
    \item[(K3)] The left action $\widetilde \lact$ satisfies (MP2) for edges.  That is, for any edges $e \in G(\Lambda)^1$ and $(g_1, g_2) \in G(\Gamma)^1 * G(\Gamma)^1,$
    \begin{align*}
        (e, r(g_1)) \widetilde \lact (s(e), [g_1g_2]) \stackrel{!}{=} (e, r(g_1)) \widetilde \lact(s(e), g_1) \left( (e, r(g_1))\widetilde \ract(s(e), g_1)\right)\widetilde \lact (s(e), g_2).
    \end{align*}
        
    \item[(K4)] We have $d(e_{\tilde \lact}(g))  = d(g)$ and $d(g_{\tilde \ract}(e)) = d(e)$.
    \item[(K5)] For any $h \in G(\Lambda)^1, f \in G(\Gamma)^1$, there are unique $e \in G(\Lambda)^1, g \in G(\Gamma)^1$ satisfying $e_{\tilde \lact}(g) = f, g_{\tilde \ract}(e) = h.$
  \end{enumerate}

\begin{proof}[Proof of Theorem \ref{thm:relaxed stable}]
   We begin by defining the actions $\widetilde \lact, \widetilde \ract$. 
   If $p\in G(\Lambda)^*, q \in G(\Gamma)^*$, we define
$ ([p], r(q))\widetilde \lact(s(p), [q]) := (r(p), [q]) $, that is, $ p_{\tilde \lact }(q) = q,$
   and $$([p], r(q)) \widetilde \ract (s(p), [q]) = ([(\xi^{s(p)}_\lact(q))_\ract(p)], s(q)), \quad \text{i.e.,} \quad q_{\tilde \ract}(p) =(\xi^{s(p)}_\lact(q))_\ract(p). $$
Equivalently, $([p], r(q)) \widetilde \ract(s(p), [q]) = ([p], r(q)) \ract(s(p), [\xi^{s(p)}_\lact(q)]).$ 

As $p_{\tilde \lact} = \id$, it is evident that $\widetilde \lact$ is a left action.  The fact  that $\widetilde \ract$ is a right action follows from Lemma \ref{lem:xi}.  To wit, suppose $q= q^1 q^2 \in G(\Gamma)^*$.  Then Lemma \ref{lem:xi} implies that for any $v \in \Lambda^0$, 
$ \xi^{v}_\lact(q^1 q^2) = \xi^{v}_\lact(q^1) \xi^{v}_\lact(q^2).$
Moreover, since $\ract$ is a right action, 
\begin{align*}  ([p], r(q)) \ract (s(p), [q^1 q^2]) & = \left(([p], r(q))\ract(s(p), [q^1])\right) \ract(s(p), [q^2]) \\
& = ( [q^1_\ract(p)], s(q^1)) \ract(s(p), [q^2]) = ([q^2_\ract(q^1_\ract(p))], s(q^2)). 
\end{align*}
That is, $(q^1 q^2)_\ract(p) = q^2_\ract(q^1_\ract(p)).$ 
Consequently, 
\begin{align*} (q^1 q^2)_{\tilde \ract}(p) & = (\xi^{s(p)}_\lact(q^1 q^2))_\ract(p) = (\xi^{s(p)}_\lact(q^1) \xi^{s(p)}_\lact(q^2))_\ract(p)\\
&= (\xi^{s(p)}_\lact(q^2))_\ract\left( (\xi^{s(p)}_\lact(q^1))_\ract(p)\right) = q^2_{\tilde \ract}(q^1_{\tilde \ract}(p)).
\end{align*}
We conclude that 
\begin{align*} 
([p], r(q))\widetilde \ract(s(p), [q^1 q^2]) & = 
([q^2_{\tilde \ract}(q^1_{\tilde \ract}(p))], s(q))= ([q^1_{\tilde \ract}(p)], s(q^1)) \widetilde \ract( s(p), [q^2]) \\
& = \left( ([p], r(q^1))\widetilde \ract (s(p), [q^1])\right) \widetilde \ract (s(p), [q^2]).
\end{align*}
That is, $\widetilde \ract$ is a right action as claimed.
   
Furthermore,    as $ 
    \ract$ preserves range and source, we will have $s(q_{\tilde \ract}(p)) = s(p)$ and $r(q_{\tilde \ract}(p)) = r(p).$

    We now check that $\widetilde \lact, \widetilde \ract $ satisfy (K1)-(K5). As $f_\lact, h_\ract$ preserve source and range for any $f\in G(\Lambda)^1, h\in G(\Gamma)^1,$ (K1) holds.  For (K2), suppose $e_1, e_2 \in G(\Lambda)^1$ satisfy $s(e_1) = r(e_2)$.  
    We need to show that $$([e_1e_2], r(g)) \widetilde \ract (s(e_2), g) \stackrel{!}{=} (e_1, r(g))\widetilde \ract \left((e_2, r(g)) \widetilde \lact(s(e_2), g)\right) \left( (e_2, r(g)) \widetilde \ract(s(e_2), g)\right).$$
   
    Observe first that $((e_2)_\lact ( \xi^{s(e_2)}_\lact(g))) = \xi^{s(e_1)}_\lact(g)$, since $r(e_2 ) = s(e_1)$ and $\lact$ is a left action (respects composition). 
    Thus, since $\ract$ satisfies (MP3), 
    \begin{align*}  ([e_1 e_2]& , r(g))  \widetilde \ract(s(e_2), g) 
     = ([e_1 e_2], r(g)) \ract(s(e_2), \xi^{s(e_2)}_\lact(g)) 
    \\
    & = \left( (e_1, r(g)) \ract ((e_2, r(g)) \lact(s(e_2), \xi^{s(e_2)}_\lact(g)) \right) (e_2, r(g)) \ract(s(e_2), \xi^{s(e_2)}_\lact(g))\\
    &=\left( (e_1, r(g)) \ract( s(e_1), \xi^{s(e_1)}_\lact(g)) \right) (g_{\tilde \ract}(e_2), s(g))\\
    &= (g_{\tilde \ract}(e_1), s(g))(g_{\tilde \ract}(e_2), s(g)).
    \end{align*}
On the other hand, since $\widetilde \lact$ is the identity, we also have 
\begin{align*}
    (e_1, r(g)) & \widetilde \ract\left( (e_2, r(g)) \widetilde \lact (s(e_2), g)\right)\left( (e_2, r(g))\widetilde \ract(s(e_2), g)\right) \\
    & =  \left((e_1, r(g))\widetilde \ract(r(e_2), g) \right) ((\xi^{s(e_2)}_\lact(g))_\ract(e_2), s(g)) \\
    &= (g_{\tilde \ract}(e_1), s(g))(g_{\tilde \ract}(e_2), s(g)).
\end{align*}
That is, (K2) holds.     
    (K3) holds since the left action $\widetilde \lact$ is always the identity, and (K4) holds since it holds for $ 
    \ract$. 
    
    For (K5), 
    given $g \in G(\Gamma)^1, h \in G(\Lambda)^1$, we need to show that there exist unique $g' \in G(\Gamma)^1, e \in G(\Lambda)^1$ so that $e_{\tilde \lact}(g') = g$ and $g'_{\tilde \ract}(e) = h.$ The first condition forces $g' = g$; as $\widetilde \ract$ preserves source, range, and degree, the candidates for $e$ are the edges in $S:= r(h) \Lambda^{d(h)} s(h).$
Thus, every edge $e\in S$ yields the same $\tilde g : = \xi^{s(e)}_\lact(g).$  In fact, our hypothesis that $p_\lact$ only depends on $s(p), r(p)$ implies that $e_\lact(\tilde g) =: \hat g$ is independent of the choice of $e\in S.$  

In other words, for any $e \in S$, when we use the factorization property 
to write $(e, r(g))(s(e), \tilde g) = (r(e), e_\lact(\tilde g))( \tilde g_\ract(e), s(g))$, the edges $(s(e), \tilde g)$ and $(r(e), e_\lact(\tilde g))$ are independent of the choice of $e \in S.$  Consequently, the factorization property yields a bijection between $\{ (e, r(g)): e \in S\}$ and $\{ (\tilde g_\ract(e), s(g)): e \in S\}.$ 
Observe furthermore that 
\[ \tilde g_\ract(e) = (\xi^{s(e)}_\lact(g))_\ract(e) = g_{\tilde \ract}(e)\]
has the same source,  range and degree as $e$ (and hence $h$).  In other words, every edge  in $S$ appears as $ \tilde g_\ract(e) = g_{\tilde \ract}(e)$ for exactly one $e \in S.$  Since $h \in S$, this proves that given $h \in G(\Lambda)^1, g \in G(\Gamma)^1$, there are unique $g' \in G(\Gamma)^1, e \in G(\Lambda)^1$ with $g = e_{\tilde \lact}(g'), h =g'_{\tilde \ract}(e). $  That is, (K5) holds and $\widetilde \Omega$ is a $k$-graph.

    Now, define $\Theta: \Omega \to \widetilde \Omega$ by 
    \begin{equation} 
    \Theta([p], [q]) = ([p], [(\xi^{s(p)})^{-1}_\lact(q)]).
    \label{eq:Theta}
    \end{equation}
    We claim that $\Theta$ is an isomorphism of $k$-graphs. The hypothesis that $\widehat \ell_\lact$ is bijective implies that $\Theta$ is. 
     Since $\widetilde \Omega_1 = \Omega_1$ by construction, and we know that $\lact$ is range-, degree-, and source-preserving, and is multiplicative by Lemma \ref{lem:xi},  $\Theta$ is a homomorphism when restricted to $\Omega_1$ or $\Omega_2.$

    To see that $\Theta$ is a $k$-graph homomorphism, we recall the product structure in $\Omega:$ 
    \begin{align*} ([p], [q]) &  = ([p], r(q)) (s(p), [q]) = \left( ([p], r(q)) \lact (s(p), [q])\right) \left( ([p], r(q))\ract (s(p), [q]) \right) \\
    &= (r(p), [p_\lact(q)]) ([q_\ract(p)], s(q)).
    \end{align*}
    Thus,  
    to see that $\Theta$ is a $k$-graph homomorphism it suffices to check that 
    \[ \Theta(r(p), [p_\lact(q)]) =(r(p), [(\xi^{r(p)})_\lact^{-1}(p_\lact(q))])  \]
equals 
\begin{align*}  \Theta([p], r(q)) \widetilde \lact \Theta(s(p), [q])) 
& = ([p], r(q)) \widetilde \lact (s(p), [(\xi^{s(p)})_\lact^{-1}(q)]) = (r(p), [(\xi^{s(p)})_\lact^{-1}(q)])
\end{align*}
and that $\Theta([q_\ract(p)], s(q)) = ([q_\ract(p)], s(q))$ equals 
\begin{align*}  \Theta([p], r(q))\widetilde \ract \Theta(s(p), [q])) & = ([p], r(q)) \widetilde \ract (s(p), [(\xi^{s(p)})_\lact^{-1}(q)]) \\
& = ([\left( \xi^{s(p)}_\lact(\xi^{s(p)})_\lact^{-1}(q)\right)_\ract(p)], s(q)) = ([q_\ract(p)], s(q)).
\end{align*}
    That is, the second equality holds because $\lact$ is a left action.

    The first equality holds thanks to our hypothesis that the function $\widehat \ell_\lact$ is bijective and completely determined by $s(\ell), r(\ell)$.  Because of this,  $(\xi^{r(p)})^{-1}_\lact \circ p_\lact = (\xi^{s(p)})^{-1}_\lact$, which yields the first equality. That is, $\Theta$ is a bijective $k$-graph homomorphism. As $\widetilde \Omega$ is stable,  this concludes the proof. 

\end{proof}

\begin{remark} 
We compute the isomorphism $\Theta$  of Theorem \ref{thm:relaxed stable} in the setting of Example \ref{ex: path loops}.  
Recall from Equation \eqref{eq:Theta} that $\Theta([p], [q]) = ([p], [(\xi^{s(p)})_\lact^{-1}(q)]).$

For each vertex $w_i \in \Lambda^0,$ let $\xi^{w_i} =: \xi^i = (e_0 \cdots e_{i-1})^{-1}$, so $(\xi^i)^{-1}_\lact = (e_{0})_\lact \circ (e_1)_\lact  \circ \cdots \circ (e_{i-1})_\lact.$  The factorization rule in $\Omega$ implies that $(e_j)_\lact $ is nontrivial (that is, interchanges $f$ and $g$) precisely when $j$ is odd.  In other words, if $i = 0, 1 \pmod 4$, then $(\xi^{i})^{-1}_\lact$ is the composition of (several identity maps with) an even number of maps which interchange $f$ and $g$, and consequently $(\xi^{i})^{-1}_\lact = id.$  On the other hand, if $i = 2,3 \pmod 4,$ then $\xi^{i}_\lact $ is the composition of an odd number of nontrivial isomorphisms (interspersed with identity maps), and so $\xi^{i}_\lact $ interchanges $f$ and $g$.   We conclude that for $\eta = f, g,$
\[ \Theta(w_i, \eta) = \psi(w_i, \eta); \]
that is, the isomorphism  
of Theorem \ref{thm:relaxed stable} is precisely the isomorphism $\psi$ of Example \ref{ex: path loops}.

\end{remark}

As we establish in the following Theorem, relaxed stability yields a complete characterization of when a quasi-product is isomorphic to a product graph (and hence has a tensor-product $C^*$-algebra).

\begin{thm}
\label{thm:product-isom}
    Let $\Omega$ be a quasi-product with connected quasi-factors $\Lambda$ and $\Gamma$. We have $\Omega \cong \Lambda \times \Gamma$ iff  $\Lambda, \Gamma$ are both relaxed stable quasi-factors of $\Omega$.
\end{thm}
\begin{proof}
    If $\Lambda, \Gamma$ are both relaxed stable quasi-factors of $\Omega$, then applying Theorem \ref{thm:relaxed stable} twice (once for each factor) yields the desired isomorphism.  On the other hand, if $\theta: \Omega \to \Lambda \times \Gamma$ is an isomorphism, we need to show that for any $p \in \widehat G(\Lambda)^*, q \in \widehat G(\Gamma)^*, \ \widehat 
    p_\lact, \widehat 
    q_\ract$ are bijective and only depend on $r(p), s(p)$ [$r(q), s(q)$ respectively].

We will first show that, thanks to our hypothesis that $\Lambda, \Gamma$ are connected,  $\theta$ restricts to isomorphisms $\theta_x: \Omega_2^x \to \Gamma, \ \theta^w: \Omega_1^w \to \Lambda$ for each $x \in \Lambda^0, w \in \Gamma^0$. As $\theta: \Omega \to \Lambda \times \Gamma$ is an isomorphism, it suffices to show that $\theta$ sends all elements of $\Omega_2^x$ to $\{y\} \times \Gamma$ for a unique $y \in \Lambda^0$, and similarly that $\text{Im}(\theta(\Omega_1^w)) \subseteq \Lambda \times \{ z\}$ for a unique $z \in \Gamma^0.$  

Fix $q, q'\in G(\Gamma)^*$ and $x \in \Lambda^0$, and suppose $\theta(x, [q]) \in \{ y\} \times \Gamma$ and $\theta(x, [q']) \in \{ u\} \times \Gamma$.  We will show $u = y$. 
Since $\Gamma$ is connected, there exists a path $e^1_\pm  \cdots  e^n_\pm \in \widehat G(\Gamma)^*$ with $r(e^1_\pm) = s(q)$ and $s(e^n_\pm) = s(q').$  As $\theta$ is a degree-preserving isomorphism, for each $i$, 
$\theta (x, e^i) \in \{ x_i\} \times G(\Gamma)^1$ for some $x_i \in \Lambda^0$.    We therefore have 
\[ s(\theta(x, e^i)), r(\theta(x, e^i)) \in \{ x_i\} \times \Gamma^0 \text{ for all } i.\]
Since $\theta$ commutes with the source and range maps, the fact that $s(e^i_\pm) = r(e^{i+1}_\pm)$ implies that  for all $i < n,$
we have 
$$ \theta( s(x, e^i_\pm) ) \in \{ s(\theta(x, e^{i+1})), r(\theta(x, e^{i+1}))\} \subseteq \{ x_{i+1}\} \times \Gamma^0. $$   Similarly,   $\{ x_n\} \times \Gamma^0 \ni \theta (s(x, e^n_\pm))  = s(\theta(x, [q'])) \in \{ u\} \times \Gamma^0 $
and $\{ x_1\} \times \Gamma^0 \ni \theta(r(x, e^1_\pm)) = s(\theta(x, [q])) \in  \{ y\} \times\Gamma^0.$ It follows that 
$$ y = x_1 = \cdots = x_n = u$$ as desired.  The same argument will show that, for each $w \in \Gamma^0,$ $\text{Im}(\theta(\Omega_1^w)) \subseteq \Lambda \times \{z\}$ for a single $z \in \Gamma^0$. 

Our next goal is to prove that for any $\xi \in \widehat G(\Lambda)^*$, $\xi_\lact$ only depends on $r(\xi), s(\xi)$. Define $\tau_{v,w}: \{ w\}\times \Gamma \to \{v\}\times\Gamma$ such that $(w,\gamma)\mapsto (v,\gamma)$;  
if $\text{Im} \theta_x = \{ y\} \times \Gamma$, we write $
\tilde x := y \in \Lambda^0$.
It is enough to show that for any  $\xi\in \widehat G(\Lambda)^*$ with $s(\xi)=w$ and $r(\xi)=v$, we have  $\theta_{v} \widehat {\xi}_\lact=\tau_{\tilde v,\tilde w}\theta_{ w}$. 
To prove this, we set $\xi = p^1_\pm \cdots p^n_\pm$ with $p^i\in G(\Lambda)^*$ and induct on $n$. 

For the case $n=1$, we first assume $\xi = p^1$.  Fix $q \in G(\Gamma)^*$, and write $(\tilde w, \gamma) := \theta(w, [q]).$ Observe that 

\begin{align*}\theta((v, [p
^1_\lact(q)])([q_\ract(p^1)], s(q)))   & =\theta(([p^1], r(q))(w, [q])) \\ &= \theta([p^1], r(q)) (\tilde w, \gamma) =: (\lambda, y)(\tilde w, \gamma) \\
& = (r(\lambda), \gamma)(\lambda, s(\gamma)),\end{align*}
thanks to the factorization rule in $\Lambda \times \Gamma$.  Moreover, the factorization property implies that 
\[ \tilde v \times \Gamma \ni \theta(v, [p^1_\lact(q)]) = (r(\lambda), \gamma);\]
that is, $r(\lambda) = \tilde v $ and $\theta_v p^1_\lact(w, [q]) = (\tilde v, \gamma) = \tau_{\tilde v, \tilde w}\theta_w(w, [q])$, as desired.  In particular, as $\tau_{\tilde v, \tilde w}, \theta_w, \theta_v$ are $k$-graph isomorphisms, we can write $$\widehat p_\lact^1=\theta_v^{-1} \tau_{\tilde{v},\tilde{w}}\theta_w.$$

We pause briefly to observe that for any $p\in G(\Gamma)^*$ the equality above ensures bijectivity of $\widehat p_\lact$.
In particular, $ (\widehat{p}_\lact)^{-1}$ is well defined and also depends only on $s(p^{-1}), r(p^{-1}).$

Having established the base case, suppose now that for any path $\eta = m^1_\pm \cdots m^j_\pm \in \widehat G(\Lambda)^*$ with $j \leq n-1$, we have $\theta_{r(\eta)} \eta_\lact = \tau_{\widetilde{r(\eta)}, \widetilde{s(\eta)}} \theta_{s(\eta)}$. Given $\xi = p^1_\pm \cdots p^n_\pm \in \widehat G(\Lambda)^*$, write $\eta = p^2_\pm \cdots p^n_\pm$ and $u = r(\eta) = s(p^1_\pm) $. The inductive hypothesis yields  $(\widehat{p^1_\pm})_\lact= \theta_v^{-1} \tau_{\tilde v,\tilde u} \theta_u$ and $\widehat\eta_\lact= \theta_u^{-1} \tau_{\tilde u,\tilde w}\theta_w$. Since $\widehat\xi_\lact = (\widehat{p}^{1}_\pm)_\lact \circ \widehat\eta_\lact$, 
\[\widehat\xi_\lact =  (\widehat p^{1}_\pm)_\lact \circ \widehat\eta_\lact = \theta_v^{-1} \tau_{\tilde v, \tilde u} \theta_u\theta_u^{-1} \tau_{\tilde u,\tilde w}\theta_w = \theta_v^{-1}\tau_{\tilde v,\tilde w}\theta_w.\]

A symmetric argument 
will show that, if we write $\tau^{x,y}: \Lambda \times \{y\} \to \Lambda \times \{x\}$ for the map sending $(\lambda, y) $ to $(\lambda, x)$, then for any $\zeta \in \widehat G(\Gamma)^*$, 
\[ \theta^{{s(\zeta)}} \widehat{\zeta}_\ract = \tau^{\widetilde{s(\zeta)}, \widetilde{r(\zeta)}} \theta^{r(\zeta)}.\]
In particular, $\widehat \zeta_\ract$ is bijective and only depends on $s(\zeta), r(\zeta).$ \qedhere 

\end{proof}

\subsection{Applications: polytrees and $k$-trees}
In some situations, the structure of  one of the quasi-factors $\Gamma$ can force any quasi-product with $\Gamma$ as a quasi-factor to be a product $k$-graph.  For example, this occurs if $\Gamma$ is a polytree or a product of polytrees (see Theorem \ref{thm: polytree} below).

\begin{dfn}
\label{def: polytree}
\cite[$\S 9.6$]{deo-graphs} 
A {\em polytree} is a directed graph $E$ such that whenever $\xi, \xi'\in \widehat E$ satisfy $s(\xi) = s(\xi'), r(\xi ) = r(\xi')$, then $\xi = \xi'$.
\end{dfn}

\begin{remark}
    \label{rem: no-squares}
    The obvious generalization of Definition \ref{def: polytree} to $k$-graphs is problematic, because any commuting square $gh \sim h'g'$ in $\Lambda$ yields a cycle $\xi := h'g' h^{-1} g^{-1}$ in $\widehat G(\Lambda)$ with $r(\xi) = r(g) = s(\xi) $ but $\xi \not= s(g)$.

\end{remark}

\begin{thm}
\label{thm: polytree}
    If $\Omega$ is a quasi-product with quasi-factors $\Lambda$ a $k$-graph and $E$ a polytree, then $\Omega\cong \Lambda\times E$.
\end{thm}

\begin{proof}
    We begin by demonstrating that $E$ is a stable quasi-factor. To this end, we appeal to Lemma \ref{lem: stab-guarantee}. 
    Let $\sigma$ 
    be a vertex fixing graph automorphism of $E$. Then $s(\sigma(e))=s(e)$ and $r(\sigma(e))=r(e)$ for all $e\in E^1$, and so $e=\sigma(e)$ and thus $\sigma=\id$. Thus, Lemma \ref{lem: stab-guarantee} forces $E$ to be a stable quasi-factor.

    We now demonstrate that $\Lambda$ is relaxed stable. To this end, let $\xi,\xi'\in \widehat{E}$ such that $s(\xi)=s(\xi')$ and $r(\xi)=r(\xi')$. 
    Then $\xi=\xi'$ and in particular $\xi_\ract=\xi'_\ract$. 
    Furthermore, Theorem \ref{thm: rho isom} guarantees that $\widehat q_\ract$ is a bijection for all $q \in \widehat E$.
    We conclude that $\Lambda$ is relaxed stable. 

    Using Theorem \ref{thm:relaxed stable}, we obtain an isomorphism $\Theta: \Omega \to \widetilde{\Omega}$ such that $\Lambda$ and $E$ are stable quasi-factors. By Theorem \ref{thm:product-isom}, $\widetilde{\Omega}\cong \Lambda\times E$, and hence $\Omega \cong \Lambda \times E$.
\end{proof}

We may appeal to the associativity of graph products to generalize polytrees to the higher-rank graph setting in a more appropriate way. 

\begin{dfn}
    \label{def: k-trees}
    Let $\Gamma$ be a $k$-graph such that $\Gamma\cong \prod E_i$. If each $E_i$ is a polytree, then we call $\Gamma$ a $k$-tree.
\end{dfn}

One notable $k$-tree is the lattice $\Omega_k=\{(m,n)\in\N^k\times \N^k: m\leq n\}$ with $d(m,n)=n-m$, $s(m,n)=n$, and $r(m,n)=m$ from \cite[Example 1.7.ii]{kp}, as we detail in the next example.  

\begin{example}
    \label{ex: lattice}
    Let $E$ be the digraph with vertices $\{w_i\}_{i\in \N}$ and edges $\{e_i\}_{i\in \N}$ with $s(e_i)=w_{i+1}$ and $r(e_i)=w_i$. Observe that paths in $E^*$ are in bijection with  pairs of integers $(j, \ell)$ with $j \leq \ell$; the source of the associated path $\omega_{(j,\ell )}:= e_j\cdots e_{\ell-1}$ in $E^*$ is $w_\ell$ and the range is $w_j$. Fix some rank $k$, and define $\Omega:=\prod_{i=1}^k E^*$. An element $m = (m_1, \ldots, m_k)\in \N^k$ corresponds  to the vertex $w_m:=(w_{m_1},\cdots, w_{m_k})\in \Omega^0$.  $w_n\in\Omega^0$. 
    
    If $n\geq m$,  
    define $\omega_{(m,n)}:=(\omega_{(m_1,n_1)},\cdots,\omega_{(m_k,n_k)})\in \Omega$. The functor 
    $(m,n)\mapsto \omega_{(m,n)}$ yields an isomorphism $\Omega_k\cong \Omega$.
\end{example}

\begin{cor}
    \label{cor: k-trees}
    If $\Omega$ is a quasi-product with quasi-factors $\Lambda$and $\Gamma$, with 
    $\Gamma$ a $k_2$-tree, then $\Omega\cong \Lambda\times\Gamma$.
\end{cor}

\begin{proof}
    We prove this by induction on $k_2$. The base case is assured by Theorem \ref{thm: polytree}. We now suppose the corollary holds for $(k_2-1)$-trees. 

    Let $\Gamma$ be a $k_2$-tree. We write $\Gamma\cong \Gamma'\times E$ for some $(k_2-1)$-tree $\Gamma'$ and some polytree $E$. Next, we leverage associativity of Cartesian products to observe that $\Omega$ may be viewed as a quasi-product with quasi-factors $\Omega'$ and $E$, where $\Omega'$ is a quasi-product with quasi-factors $\Lambda$ and $\Gamma'$. Theorem \ref{thm: polytree} assures us that $\Omega\cong \Omega'\times E$. 
    
    The injections $\psi_1:\Lambda\into \Omega$ and $\psi_2:\Gamma\into \Omega$ may be restricted to injections $\psi_1:\Lambda \into \Omega'$ and $\psi_2: \Gamma' \into \Omega'$. That is, the quasi-factor $\Omega'$ is itself a quasi-product with quasi-factors $\Lambda$ and $\Gamma'$. We then use the inductive hypothesis to conclude that $\Omega'\cong \Lambda\times \Gamma'$ giving us $\Omega\cong (\Lambda\times\Gamma')\times E = \Lambda\times (\Gamma'\times E)=\Lambda\times \Gamma$.
\end{proof}

\bibliography{products}
\bibliographystyle{plain}

    \end{document}